\documentclass[a4paper, reqno, 11pt]{amsart}
\usepackage[usenames,dvipsnames]{color}
\usepackage{amsthm,amsfonts,amssymb,amsmath,amsxtra, array}
\usepackage[all]{xy}
\usepackage{xr-hyper}
\usepackage{hyperref}
\usepackage[normalem]{ulem}
\usepackage{soul}
\usepackage{dynkin-diagrams}
\usepackage{verbatim}
\usepackage[cal=boondox]{mathalfa}

\usepackage{tikz-cd, epic}
\usepackage{bbm}
\usepackage[margin=1.25in]{geometry}
\usepackage{mathrsfs}
\usepackage{BOONDOX-calo}
\usepackage{float}
\restylefloat{table}
\usepackage{MnSymbol}
\usepackage{stmaryrd}

\RequirePackage{xspace}
\RequirePackage{etoolbox}
\RequirePackage{varwidth}
\RequirePackage[shortlabels]{enumitem}
\RequirePackage{mathtools}
\RequirePackage{longtable}

\newcommand\iso{\stackrel{\sim}{\smash{\longrightarrow}\rule{0pt}{0.4ex}}}

\def\le{\leqslant}

\def\s{\sigma}

\def\<{\langle}
\def\>{\rangle}

\newcommand{\fkp}{\ensuremath{\mathfrak{p}}\xspace}

\newcommand{\fks}{\ensuremath{\mathfrak{s}}\xspace}
\newcommand{\fkt}{\ensuremath{\mathfrak{t}}\xspace}

\newcommand{\fkB}{\ensuremath{\mathfrak{B}}\xspace}

\newcommand{\fkI}{\ensuremath{\mathfrak{I}}\xspace}

\newcommand{\fkM}{\ensuremath{\mathfrak{M}}\xspace}

\newcommand{\fkO}{\ensuremath{\mathfrak{O}}\xspace}

\newcommand{\fkR}{\ensuremath{\mathfrak{R}}\xspace}
\newcommand{\fkS}{\ensuremath{\mathfrak{S}}\xspace}

\newcommand{\fkZ}{\ensuremath{\mathfrak{Z}}\xspace}

\newcommand{\heart}{{\heartsuit}}

\newcommand{\spade}{{\spadesuit}}

\newcommand{\bG}{\mathbf G}

\newcommand{\bT}{\mathbf T}

\newcommand{\bM}{\mathbf M}

\newcommand{\BC}{\ensuremath{\mathbb {C}}\xspace}

\newcommand{\BF}{\ensuremath{\mathbb {F}}\xspace}
\newcommand{{\BG}}{\ensuremath{\mathbb {G}}\xspace}

\newcommand{{\BK}}{\ensuremath{\mathbb {K}}\xspace}

\newcommand{\BN}{\ensuremath{\mathbb {N}}\xspace}

\newcommand{\BQ}{\ensuremath{\mathbb {Q}}\xspace}
\newcommand{\BR}{\ensuremath{\mathbb {R}}\xspace}

\newcommand{\BZ}{\ensuremath{\mathbb {Z}}\xspace}

\newcommand{\CA}{\ensuremath{\mathcal {A}}\xspace}
\newcommand{\CB}{\ensuremath{\mathcal {B}}\xspace}

\newcommand{\CH}{\ensuremath{\mathcal {H}}\xspace}
\newcommand{\CI}{\ensuremath{\mathcal {I}}\xspace}
\newcommand{\CJ}{\ensuremath{\mathcal {J}}\xspace}
\newcommand{\CK}{\ensuremath{\mathcal {K}}\xspace}

\newcommand{\CO}{\ensuremath{\mathcal {O}}\xspace}

\newcommand{\CV}{\ensuremath{\mathcal {V}}\xspace}
\newcommand{\CW}{\ensuremath{\mathcal {W}}\xspace}

\newcommand{\CZ}{\ensuremath{\mathcal {Z}}\xspace}

\newcommand{\ca}{\mathbcal{a}}

\newcommand{\tF}{{\widetilde{F}}}

\newcommand{\tH}{{\widetilde{H}}}

\newcommand{\tJ}{\widetilde{J}}

\DeclareMathOperator{\End}{\text{End}}
\DeclareMathOperator{\Mod}{\text{mod}}

\newcommand{\GL}{\mathrm{GL}}

\DeclareMathOperator{\Hom}{Hom}

\newcommand{\Ind}{{\mathrm{Ind}}}
\newcommand{\ind}{{\mathrm{ind}}}

\newcommand{\Int}{\ensuremath{\mathrm{Int}}\xspace}

\DeclareMathOperator{\Irr}{Irr}

\DeclareMathOperator{\Ker}{Ker}

\DeclareMathOperator{\Res}{Res}

\DeclareMathOperator{\vol}{vol}

\newcommand{\bS}{{\breve{\bf S}}}

\newtheorem{theorem}{Theorem}
\newtheorem{proposition}[theorem]{Proposition}
\newtheorem{lemma}[theorem]{Lemma}

\newtheorem{corollary}[theorem]{Corollary}

\theoremstyle{definition}
\newtheorem{definition}[theorem]{Definition}

\newtheorem{remark}[theorem]{Remark}

\newtheorem{assumption}[theorem]{Assumption}

\def\nd{\mathrm{nd}}

\numberwithin{equation}{section}
\numberwithin{theorem}{section}

\setitemize[0]{leftmargin=*,itemsep=\the\smallskipamount}
\setenumerate[0]{leftmargin=*,itemsep=\the\smallskipamount}

\renewcommand{\to}{%
   \ifbool{@display}{\longrightarrow}{\rightarrow}%
   }
\let\shortmapsto\mapsto
\renewcommand{\mapsto}{%
   \ifbool{@display}{\longmapsto}{\shortmapsto}%
   }
\newlength{\olen}
\newlength{\ulen}
\newlength{\xlen}
\newcommand{\xra}[2][]{%
   \ifbool{@display}%
      {\settowidth{\olen}{$\overset{#2}{\longrightarrow}$}%
       \settowidth{\ulen}{$\underset{#1}{\longrightarrow}$}%
       \settowidth{\xlen}{$\xrightarrow[#1]{#2}$}%
       \ifdimgreater{\olen}{\xlen}%
          {\underset{#1}{\overset{#2}{\longrightarrow}}}%
          {\ifdimgreater{\ulen}{\xlen}%
             {\underset{#1}{\overset{#2}{\longrightarrow}}}
             {\xrightarrow[#1]{#2}}}}%
      {\xrightarrow[#1]{#2}}
   }
\makeatother
\newcommand{\xyra}[2][]{%
   \settowidth{\xlen}{$\xrightarrow[#1]{#2}$}%
   \ifbool{@display}%
      {\settowidth{\olen}{$\overset{#2}{\longrightarrow}$}%
       \settowidth{\ulen}{$\underset{#1}{\longrightarrow}$}%
       \ifdimgreater{\olen}{\xlen}%
          {\mathrel{\xymatrix@M=.12ex@C=3.2ex{\ar[r]^-{#2}_-{#1} &}}}%
          {\ifdimgreater{\ulen}{\xlen}%
             {\mathrel{\xymatrix@M=.12ex@C=3.2ex{\ar[r]^-{#2}_-{#1} &}}}
             {\mathrel{\xymatrix@M=.12ex@C=\the\xlen{\ar[r]^-{#2}_-{#1} &}}}}}%
      {\mathrel{\xymatrix@M=.12ex@C=\the\xlen{\ar[r]^-{#2}_-{#1} &}}}%
   }
\makeatletter
\newcommand{\xla}[2][]{%
   \ifbool{@display}%
      {\settowidth{\olen}{$\overset{#2}{\longleftarrow}$}%
       \settowidth{\ulen}{$\underset{#1}{\longleftarrow}$}%
       \settowidth{\xlen}{$\xleftarrow[#1]{#2}$}%
       \ifdimgreater{\olen}{\xlen}%
          {\underset{#1}{\overset{#2}{\longleftarrow}}}%
          {\ifdimgreater{\ulen}{\xlen}%
             {\underset{#1}{\overset{#2}{\longleftarrow}}}
             {\xleftarrow[#1]{#2}}}}%
      {\xleftarrow[#1]{#2}}
   }
\newcommand{\isoarrow}{%
   \ifbool{@display}{\overset{\sim}{\longrightarrow}}{\xrightarrow\sim}%
   }
 \newcommand{\bF}{{\breve{F}}}

\newcommand{\ba}{{\breve{a}}}
\newcommand{\bc}{{\breve{c}}}

\newcommand{\bal}{{\breve{\mathbcal{a}}}}

\newcommand{\nr}{\mathrm{nr}}

\begin{document}
\title[]{The center of Hecke algebras of types}

\author[Reda Boumasmoud]{Reda Boumasmoud}
\address{Institut für Mathematik
Universität Zürich
Winterthurerstrasse 190
CH-8057 Zürich, Switzerland}
\email{reda.boumasmoud@math.uzh.ch}

\author{Radhika Ganapathy}
\address{Department of Mathematics, Indian Institute of Science, Bengaluru, Karnataka  - 560012, India.}
\email{radhikag@iisc.ac.in}

\keywords{$p$-adic groups, Hecke algebras, types, Bernstein decomposition}

\subjclass[2010]{11F70, 22E50}
\begin{abstract} We describe the center of the Hecke algebra of a type attached to a Bernstein block under some hypothesis. When $\bf G$ is a connected reductive group over non-archimedean local field $F$ that splits over a tamely ramified extension of $F$ and the residue characteristic of $F$ does not divide the order of the absolute Weyl group of $\bf G$, the works of Kim-Yu and Fintzen associate a type to each Bernstein block and our hypothesis is satisfied for such types. We use our results to give a description of the Bernstein center of the Hecke algebra $\CH({\bf G } (F),K)$ when $K$ belongs to a nice family of compact open subgroups of ${\bf G }(F)$ (which includes the Moy--Prasad filtrations of an Iwahori subgroup) via the theory of types.
\end{abstract}
\maketitle
\section*{Introduction}
Let $F$ be a non-archimedean local field. For a connected, reductive group $\bf G$ over $F$, we write $G$ for its $F$-points.

Let $\fkR(G)$ denote the category of smooth, complex representations of $G$.  Let $\fkB(G)$ for the set of all inertial equivalence classes in $G$ (this definition is recalled in Section \ref{BD}). The Bernstein decomposition yields

\[\fkR(G) = \prod_{\fks \in \fkB(G)} \fkR^\fks(G).\]

 We are interested in describing the center of $\fkR^\fks(G)$, $\fks \in \fkB(G)$. Let $J$ be a compact open subgroup of $G$ and let $\rho$ be an irreducible representation of $J$ such that $(J,\rho)$ is an $\fks$-type (see Definition \ref{def:type}). Then the category $\fkR^\fks(G)$ is equivalent to $\CH(G, \rho)-\text{mod}$. This leads us to the question of understanding the center of Hecke algebras of types. 
 
 We take us this question for supercuspidal blocks in Section \ref{sometypessec}. First, suppose $\pi$ is an irreducible supercuspidal representation of $G$ of the form $\ind_{\tJ}^G \tilde\rho$, where $\tJ$ is an open, compact mod center subgroup of $G$ and $\tilde\rho$ is an irreducible representation of $\tJ$. Let $^0G$ be the open normal subgroup of $G$ as in \eqref{G0} and let $J = ^0G \cap \tJ$ and let $\rho$ be an irreducible summand of $\Res^{\tilde J}_{J}\tilde\rho$. Then $(J, \rho)$ is an $\fks = [G,\pi]_G$-type.  Assume that the intertwiners of $\rho$, denoted $\CI_G(\rho)$, is contained in $ \tJ$. These requirements are satisfied for supercuspidal representations arising out of Yu's construction (see \cite{Yu01}), which exhaust all supercuspidal representations of $G$ by \cite{Fin21} under the hypothesis that $\bf G$ splits over a tamely ramified extension of $F$ and the residue characteristic of $F$ does not divide the order of the Weyl group of $\bf G$. Let $\pi_0$ be an irreducible summand of $\Res^{G}_{{}^0 G}\pi$. We show in Theorem \ref{typesmultiplicites} that multiplicity with which $\pi_0$ occurs in $\Res^{G}_{{}^0 G}\pi$  is equal to the multiplicity with which $\rho$ occurs in $\Res^{\tilde J}_{J}\tilde\rho$. This theorem allows us to translate some results about Bernstein blocks discussed in \cite{Roc09} into statements about types attached to the Bernstein block and in particular allows us to describe the center $\CH(G,\rho)$  of the type $(J,\rho)$. We show that the center $\CZ(\CH(G, \rho)) \simeq \BC[{}^\dagger J/J]$ where  ${}^\dagger J =\bigcap_{\nu \in X_{\tJ}(\tilde\rho)} \ker(\nu)$, with $X_{\tJ}(\tilde\rho) = \{ \nu \in \Hom(\tJ/J, \BC^\times) \;|\; \tilde\rho\otimes \nu\simeq \tilde\rho\}$ (See Lemma \ref{clifbis}, Lemma \ref{jdag} and Theorem \ref{centerhecke}). We deduce that the Hecke algebra $\CH(G,\rho)$ is commutative if and only if $\pi|_{^0G}$ is multiplicity free (see Proposition \ref{Prop:HAComm}). 

   In Section \ref{sec:nonsup}, we take up the question of describing the center of non-supercuspidal blocks. Let $\fks=[M, \sigma]_G$ and $\fks_M = [M, \sigma]_M$. We assume that $\sigma$ is an irreducible supercuspidal representation of $M$ of the form  $\ind_{\tJ_M}^M \tilde\rho_M$, where $\tJ_M$ is an open, compact mod center subgroup of $M$ and $\tilde\rho_M$ is an irreducible representation of $\tJ_M$. Let $(J_M, \rho_M)$ be the $\fks_M$-type as before.  Again we assume that $\CI_G(\rho_M) \subset \tJ_M$. Let $(J,\rho)$ be a $G$-cover of $(J_M, \rho_M)$. Then $(J,\rho)$ is an $\fks$-type. After proving some preparatory technical results in Sections \ref{gequiv} and \ref{sec:nortype}, we show in Section \ref{WeylAcCEnt} that the center $\CZ(\CH(G,\rho)) \simeq \BC[{}^\dagger J_M/J_M]^{W(\rho_M)}$ where $W(\rho_M)$ is described in Proposition \ref{Gequiv} (see Theorem \ref{Satakegen}).

We note here that all the results in Section \ref{sometypessec}, Section \ref{gequiv}, and Section \ref{sec:nortype} hold when ${}^0 G$ is replaced by a more general open normal sugroup ${}^\flat G$ defined in Section \ref{def:flat}. In fact they hold when ${}^\flat G = G(F)_1$, the kernel of the Kottwitz homomorphism. Further, Theorem \ref{genpro54}, Corollary \ref{isogflat0}, and Corollary \ref{cor:2flat-0} enable a passage between working with ${}^0 G$ and ${}^\flat G$. 
 
 Now, assume that $\bf G$ splits over a tamely ramified extension of $F$ and the residue characteristic of $F$ does not divide the order of the absolute Weyl group of $\bf G$. Then by \cite{KY17, Fin21}, every Bernstein block has a Kim-Yu type attached to it and our results in the preceeding paragraphs hold for such types. We use this to give a description of the Bernstein center of $\CH(G,K)$ for certain nice compact open subgroups $K$ of $G$. Let us describe what these compact open subgroups are.

 Let $\CB(G,F)$ denote the Bruhat--Tits building of $\bf G$ over $F$. 
Let $\bf S$ be a maximal $F$-split torus in $\bf G$ and let $\CA(S, F)$ be the apartment of $\bf S$ over $F$.  For a compact open subgroup $K$  of $G$ and let $\fkR_K(G)$ be the full sub-category of $\fkR(G)$ consisting of representations $(\pi, V)$ that are generated by their $K$-fixed vectors. In \cite[Section 3.7 - 3.9]{BD84}, the authors put criteria $\heart_S$ (see Definition \ref{heart}) on the compact open subgroup $K$ and prove that the category $\fkR_K(G)$ is closed under taking sub-quotients when $K$ satisfies these criteria. They further show that if $x$ is a special point in $\CA(S,F)$, then $G_{x,r}$ satisfies $\heart_S$ for each $r>0$. It was long expected that the category $\fkR_K(G)$ is closed under taking sub-quotients whenever $K = G_{x,r}$ for all points $x \in \CB(G,F)$ and all $r>0$. In \cite{BS20}, Bestvina--Savin put slightly different criteria $\spade_S$ (see Definition \ref{spade}) on the compact open subgroup $K$ for which the category $\fkR_K(G)$ is closed under taking sub-quotients. 
Further, they prove that $G_{x,r}$ satisfies $\spade_S$ for each $x \in \CA(S,F)$ and each $r>0$. If $K$ satisfies $\heart$ or $\spade$, then there is a finite subset $\fkS(K) \subset \fkB(G)$ such that
\[\fkR_K(G) = \prod_{\fks \in \fkS(K)} \fkR^\fks(G).\]

When $K$ satisfies $\spade_S$, it is easy to see that for a Levi subgroup $\bf M$ of $\bf G$ that contains $\bf S$, and a representation $\sigma$ of $M$, if $(\Ind_P^G \sigma)^K \neq 0$, then   $\sigma^{K'_M} \neq 0$ for a $G$-conjugate $K'$ of $K$ (that has an Iwahori factorization), where $\bf{P = MN}$ is a parabolic subgroup of $\bf G$ with Levi $\bf M$ and $K_M = K\cap P/K \cap N$. On the other hand, when $K$ satisfies $\heart_S$ it follows that if $(\Ind_P^G \sigma)^K \neq 0$, then   $\sigma^{K_M} \neq 0$. This property yields finer information about the set $\fkS(K)$ (see Lemma \ref{BuProp}). For this reason, it is helpful to know which $G_{x,r}, x \in \CA(S,F), r>0,$ also satisfy $\heart_S$. We prove two results in this direction. First, we show that if $\ca$ is an alcove in $\CA(S,F)$ and $x \in \ca$, then $G_{x,m}, m \in \BN,$ always satisfies $\heart_S$ (See Proposition \ref{PropIm} for the precise statement). Next, we give two examples of Moy-Prasad filtration subgroups that don't satisfy $\heart_S$. The first one is the Moy-Prasad filtration subgroup $G_{x,1} \subset \GL_3(F)$, where $x$ is a non-special point in the boundary of an alcove of $\CA(S,F)$. The second is the  Moy-Prasad filtration subgroup $G_{x_b, 3/8} \subset \GL_4(F)$ where $x_b$ is the barycenter of an alcove of $\CA(S,F)$ (see Section \ref{CounterEg}). 

In Section \ref{BC}, we discuss some applications of the results in the preceding sections to Yu's supercuspidal representations. Let $\pi = \ind_{\tJ}^G \tilde \rho$ be a supercuspidal representation arising out of Yu's contruction and let $(J, \rho)$ be the type attached to it as before. We show that the multiplicity with which $\rho$ occurs in $\tilde\rho$ agrees with its depth 0 counterpart that is part of the initial datum of Yu's construction; see Lemma \ref{LemmaonYu} for the precise statement. We give a simple application of this lemma in Corollary \ref{CorHAYu}. Finally, we give a description of the Bernstein center of $\CH(G,K)$ where $K$ is a compact open subgroup of $G$ that satisfies $\spade_S$ or $\heart_S$, using the theory of types.
\section*{Acknowledgements}
This project has its origin in a note that Thomas Haines had shared with the first author which indicated that one could try to use the theory of types to describe the center of Hecke algebras at deeper level. We are grateful to him for this suggestion. We thank  Xuhua He, Dipendra Prasad, Cheng-Chiang Tsai, Marie-France Vigneras, and Jiu-Kang Yu for the many helpful comments, suggestions and corrections on a previous draft of this article. 
\section{The Bernstein center}
Let $F$ be a non-archimedean local field and let $\bf G$ be a connected, reductive group over $F$. Let $G = {\bf G}(F)$. Let $\bf Z$ denote the center of $\bf G$.

In this paper, we will consider two induction functors and a restriction functor (see \cite[Section I.5.1]{Vigneras1996}):
 \begin{itemize}
     \item $\Ind$ denotes the usual induction functor.
     \item $\ind$ denotes the compact induction functor. 
     \item $\Res$ denotes the restriction functor.
 \end{itemize}

\subsection{The Bernstein decomposition}\label{BD}
 Let $X_F(G)$ be the group of $F$-rational characters $\chi: G \rightarrow F^\times$ of $G$. For $\chi \in X_F(G)$ and $s \in \BC$, we define a smooth one-dimensional representation $g \rightarrow |\chi(g)|_F^s$ of $G$. Let $X_{\nr}(G)$ be the group of unramified quasi-characters of $G$, generated by maps $G \rightarrow \BC^\times$ of this form. We write 
\begin{align}\label{G0}
    \displaystyle{^0G =  \bigcap_{\chi \in X_\nr(G)} \Ker(\chi).}
\end{align}
The quotient $G/^0G$ is free abelian of finite rank and $X_\nr(G) = \Hom(G/^0G, \BC^\times)$.

Let  $\fkR(G)$ the category of smooth, complex representations of $G$ and $\fkZ$ its center. Let ${\Irr}(G)$ the set of irreducible objects in $\fkR(G)$.

We consider pairs $(M, \sigma)$ where $\bf M$ is an $F$-Levi subgroup of $\bf G$ and $\sigma$ is a supercuspidal representation of $M$. Two pairs $(M_1, \sigma_1)$ and $(M_2, \sigma_2)$ are inertially equivalent if there exist $g \in G$ and $\chi \in X_{\nr}(M_2)$ such that 
\[{\bf M_2} = {}^g {\bf M_1} \text{ and } {}^g\sigma_1 = \sigma_2 \otimes \chi.\]
Here ${}^g {\bf M_1} = g {\bf M_1} g^{-1}$ and ${}^g\sigma_1: x \rightarrow \sigma_1(g^{-1}xg)$, for $x \in {}^g M_1$. We write $[M, \sigma]_G$ for the inertial equivalence class of the pair $(M, \sigma)$ and $\fkB(G)$ for the set of all inertial equivalence classes in $G$. 

Now, let $\pi$ be an irreducible, smooth representation of $G$. There exists an $F$-parabolic subgroup ${\bf P}$ of $\bf G$ with Levi component $\bf M$, and an irreducible, supercuspidal representation $\sigma$ of $M$ such that $\pi$ occurs as an irreducible sub-quotient of the 
normalized parabolically induced representation $\Ind_P^G(\sigma)$. 
The representation $\pi$ determines a unique inertial equivalence class $[M, \sigma]_G$ which we denote as $\fkI(\pi)$ and call it the inertial support of $\pi$ (See \cite[\S II.2.20]{Vigneras1996} for more properties). 





For $\fks \in \fkB(G)$ we define a full subcategory $\fkR^\fks(G)$ of $\fkR(G)$ as follows. Let $(\pi, V) \in \fkR(G)$. Then $(\pi, V) \in \fkR^\fks(G)$ if and only if every irreducible subquotient of $\pi$ has inertial support $\fks$. 

Let us recall some results on the Bernstein decomposition. 
\begin{enumerate}
    \item We have $\fkR(G) = \prod_{\fks \in \fkB(G)} \fkR^\fks(G).$ (see \cite[Theorem 1.7.3.1]{Roc09}).
    \item Let $\fks = [G, \pi]_G$ and let $\Irr^\fks(G)$ denote the set of isomorphism classes of irreducible objects in $\fkR^\fks(G)$. Let $\fkZ^\fks$ denote the center of $\fkR^\fks(G)$. Then $\Irr^\fks(G)$ can be endowed with the structure of a complex affine variety whose ring of regular functions can be identified with the center $\fkZ^\fks$ (see \cite[Section 1.6.3]{Roc09}). 
    \item Let $\fkt = [M, \sigma]_M \in \fkB(M)$ and let $\fks = [M, \sigma]_G \in \fkB(G)$. The action of $N_G(M)$ on $M$ by conjugation induces an action of $W(M)$ on $\fkB(M)$. Let $W_\fkt$ denote the stabilizer of $\fkt$. Thus $W_\fkt = N_\fkt/M$ where
\[N_\fkt = \{ n \in N_G(M) \;|\; ^n \sigma \simeq \sigma \nu, \text{ for some } \nu \in X_{\nr}(M)\}.\]
Let $\fkZ^\fkt$ denote the center of $\fkR^\fkt(M)$. The group $W_\fkt$ acts on $\Irr^\fkt(M)$ and hence acts on $\fkZ^\fkt$, the ring of regular functions of  $\Irr^\fkt(M)$. Let $\fkZ^\fks$ denote the center of $\fkR^\fks(G)$. Then (see \cite[Theorem 1.9.1.1]{Roc09})
\[ \fkZ^\fks = (\fkZ^\fkt)^{W_\fkt}.\]
\end{enumerate}

\subsection{Type associated to a Bernstein block}
 Let 
$\Omega(G)$ the set of open compact subgroups of $G$ and let $\Omega(G/Z)$ the set of open subgroups of $G$ containing $Z$ and compact mod $Z$. Let $\Irr(H)$ denote the set of irreducible representations of $H$ for $H \in \Omega(G)$ or $\Omega(G/Z)$.
\begin{definition}\label{def:type}
Let $J \in \Omega(G)$ and $\rho\in \Irr(J)$. For a subgroup $H$ of $G$ that contains $J$,
let $\CH(H, \rho) := \End_H(\ind_J^H(\rho))$. The pair $(J, \rho)$ is a \textit{type} in $G$ if it satisfies the following equivalent conditions (see \cite[Section 4.2]{BK98}): 
\begin{enumerate}
\item there exists a finite subset ${\fkS(J,\rho)}$ of $\fkB(G)$ such that for all $ \pi \in \Irr(G)$, $\fkI(\pi) \in {\fkS(J,\rho)} $ if and only if $\Hom_J(\rho, \pi) \neq 0$; 
\item Let $ \fkR_\rho$ denotes the category of representations of $G$ that are generated by their $\rho$-isotypic subspace. Then $\fkR_\rho$ is closed under subquotients. Further \[\fkR_\rho = \fkR^{{\fkS(J,\rho)}}(G):=\prod_{\fks \in {\fkS(J,\rho)}} \fkR^{\fks}(G);\]
\item the functor 
\[\fkM_\rho: \fkR_\rho \rightarrow \CH(G, \rho)-\Mod, \quad \pi \mapsto \Hom_J(\rho, \pi)\] is an equivalence of categories, where $\CH(G, \rho)-\Mod$ denotes the category of right-modules over $\CH(G, \rho)$. 
\end{enumerate}
\end{definition}
In this situation, the pair $(J,\rho)$ is called an \textit{${\fkS(J,\rho)}$-type}. 
\begin{definition} We say that two types $(J,\rho)$ and $(J',\rho')$ are $G$-associate and write it as $(J,\rho) \approx_G (J',\rho')$ if ${\fkS(J,\rho)}={\fkS(J',\rho')}$. The $G$-association class of a type $(J,\rho)$ will be denoted $[J,\rho]_G$. 
\end{definition}
\begin{remark}
    We will see later another  equivalence relation of (weakly cuspidal) types which will be denoted $\cong_G$ (see Definition \ref{def:wct}).
\end{remark}

\section{Some Clifford theory}
 In this section, we collect some facts from Clifford theory that will be used later in the work.  For the remainder of this paper, we will use the following notation.  
{For a group $H$ let $X(H) :=\Hom(H, \BC^\times)$. If $H'$ is a normal subgroup of $H$, we will abuse notation and identify $X(H/H')$ with the subgroup of $$\{\chi\in X(H)\,\colon \ker(\chi) \supset H'\}.$$}
If $H\subset G$ and $\sigma$ is a representation of $H$, then for any $g\in G$, ${}^g \sigma$ denote the representation of ${}^g H= g H g^{-1}$ defined by ${}^g\sigma(h)=\sigma(ghg^{-1})$, for all $h\in H$.

\subsection{$\flat$-world}\label{def:flat}
Let ${}^\flat G \trianglelefteq G$ of finite index in ${}^0G$ containing the commutator subgroup of $G^{\text{der}}$ of $G(F)$ and some open subgroup. 
So obviously, ${{}^\flat} G$ is open. In the latter part of the article, we will restrict to the case ${}^\flat G = {}^0 G$. But for instance,  one can take ${}^\flat G=\ker(\kappa_G)$ the kernel of the Kottwitz map. 
 
\subsection{} Since ${}^\flat GZ$ is a normal subgroup of finite index in $G$ and since for the representations considered below, $Z$ acts via a character, we have the following as in usual Clifford theory (cf. \cite{Cl37}).

\begin{lemma}\label{GenClif}
Let $\tilde H$ be any open subgroup of $G$ and set $H=\tilde H \cap {}^\flat G$. 
Let $\tilde\sigma$ be a semi-simple of finite length representation of $\tilde H$. 
\begin{enumerate}
\item $\Res_{ H}^{\tilde H}(\tilde\sigma)$ is semi-simple of finite length so if ${ H}$ is compact then $\dim_\BC(\tilde \sigma)< \infty$.
\item[]Assume from now on that $\tilde\sigma$ is \textbf{irreducible}: 
\item Denote by $\CO_{H}(\tilde \sigma)$ the set of all $\sigma \in {\Irr}(H)$ which are isomorphic to a sub-representation of $\Res_{ H}^{\tilde H}(\tilde\sigma)$. Write ${\Int}_{\widetilde{H}}(\sigma)=\{g\in \widetilde{H}\;|\;   \sigma \simeq  {}^g\sigma\} $ for the inertia group of $\sigma$ in $\widetilde{H}$. We have $\CO_{H}(\tilde \sigma) \simeq \tilde H/{\Int}_{\widetilde{H}}(\sigma)$.
\item Let $\sigma \in {\Irr}({ H})$ contained in $\Res_{ H}^{\tilde H}(\tilde\sigma)$. 
There exists a positive integer $m_{{ H}}(\widetilde{\sigma})$ such that
$$\Res_{{ H}}^{\widetilde{H}}(\widetilde{\sigma})\simeq m_{{ H}}(\widetilde{\sigma}) \cdot(\bigoplus_{g \in \tilde{H}/{\Int}_{\widetilde{H}}(\sigma)} {}^g\sigma)$$
and $\{{}^g\sigma \colon g \in \widetilde{H}/{\Int}_{\widetilde{H}}(\sigma)  \}$ are all nonisomorphic conjugates of $\sigma$. 
In particular, if $\dim_\BC(\sigma)<\infty$ then 
$$\dim_\BC(\tilde\sigma)=|\widetilde{H}/{\Int}_{\widetilde{H}}(\sigma)|\cdot m_{ H}(\tilde\sigma)\cdot \dim_\BC(\sigma).$$
\item Let $\widehat{\sigma}$ be the sum of all subrepresentations of $\Res_{{ H}}^{\widetilde{H}}(\widetilde\sigma)$ that are isomorphic to $\sigma$. 
Then $\widehat{\sigma}$ is an irreducible representation for ${\Int}_{\widetilde{H}}(\sigma) $ and 
$$\Res_{ H}^{{\Int}_{\widetilde{H}}(\sigma) }(\widehat{\sigma}) \simeq m_{{ H}}(\widetilde{\sigma})  \cdot \sigma \quad  \text{ and } \quad \widetilde{\sigma} \simeq \Ind_{{\Int}_{\widetilde{H}}(\sigma) }^{\widetilde{H}}(\widehat{\sigma}).$$
\item For any $\tilde\sigma'\in {\Irr}(\tilde H)$, the following properties are equivalent:
\begin{itemize}  
		\item [ i)] $\tilde\sigma' \simeq \chi\otimes \tilde\sigma$ for some character $\chi \in X(\tilde H/{ H})$; 
		\item [ ii)] $\Res_{ H}^{\tilde H}(\tilde\sigma') \simeq \Res_H^{\tilde H}(\tilde\sigma)$; 
		\item[ iii)] $\CO_{ H}(\tilde \sigma)=\CO_{ H}(\tilde \sigma')$;
		\item [ iv)] $\CO_{ H}(\tilde \sigma)$ and $\CO_{ H}(\tilde \sigma')$ has a direct irreducible factor that are isomorphic.
\end{itemize}

\end{enumerate}
\end{lemma}
\subsubsection{}

Let ${}^\dag H:= \cap_{\chi \in X_{\tilde{H}}(\tilde \sigma)} \ker (\chi)$ where $X_{\tilde H}(\tilde \sigma):=\{\chi \in X(\widetilde{H}/H)\colon \widetilde{\sigma}\otimes\chi \simeq \widetilde{\sigma}\}$. 
 This  group contains the center $Z_{\tilde H}$ of $\tilde H$ and $H$, hence it is an open normal finite index subgroup of $\tilde{H}$. 
 We choose an irreducible ${ H}$-subspace $W$ of $\tilde{\sigma}$ whose $\tilde{H}$-stabilizer ${}^s H$ is maximal. We let $^s \sigma$ denote the natural representation of ${}^s H$ on $ W$. 
 Replacing by a $\tilde{H}$-conjugate, we can assume that $\Res_{ H}^{{}^s H}({}^s \sigma)= \sigma$. 
 
\begin{lemma}\label{clifbis}
Using the previous notation, we have the following properties.
\begin{enumerate}
\item $\Res^{\tH}_{{}^s H}\tilde{\sigma}\simeq \bigoplus_{h \in \tilde H/{}^s H} ({}^s \sigma)^h$.
 \item We have ${}^\dag H \le {}^s H \le \Int_{\tilde{H}}( \sigma) $ and $[ {\Int}_{\tilde{H}}( \sigma):{}^s H]=[{}^s H:{}^\dag H]=m_{ H}(\tilde \sigma)$. 
 \item $\tilde \sigma \simeq \ind_{{}^s H}^{\tilde H}({}^s \sigma)$.
 \item The representation ${}^\dag \sigma:= \Res_{{}^\dag H}^{^s H}(^s \sigma)$ is the unique irreducible representation which occurs in $\Res_{{}^\dag H}^{\tilde H}(\tilde{\sigma})$ and satisfies $\Res_{ H}^{{}^\dag H}({}^\dag \sigma)= \sigma$. 
 \item $ \ind_{{}^\dag H}^{\tilde H}({}^\dag \sigma) \simeq m_{ H}(\tilde{\sigma}) \cdot \tilde \sigma $. 
 \item $X_{\tilde H}(\tilde \sigma) \simeq X(\tilde{H}/{}^\dag H)$. 
 \end{enumerate}
 \end{lemma}
\begin{proof} This is \cite[Lemma 8.3]{BH03} when $\tH=G$ and $ H = {}^\flat G= {}^0G$, whose detailed proof can be found in \cite[Lemma 1.6.3.1]{Roc09}. The exact same proof works out in this more general set up.
\end{proof}

\section{Some results on Hecke algebras of types} \label{sometypessec}

\begin{definition}\label{pairsCJ}
Define $\CJ_G$ to be the set of pairs $(J,\rho)$ formed by $ J\in \Omega(G)$ and an irreducible representation $\rho$ of it, such that:
\begin{itemize}
    \item There exists $\tilde J \in \Omega(G/Z)$ such that $J=\tilde J \cap {}^\flat G$ and  $\tilde \rho \in \Irr(\tilde J)$ containing $\rho$,
    \item $\CI_G(\widetilde\rho) = \tJ$, or equivalently $\pi= \ind_{\tJ}^G (\tilde\rho)$ is irreducible (hence supercuspidal).
\end{itemize}
\end{definition} 
For any $(J,\rho) \in \CJ_G$, fix $\tilde J$ as above, set ${{}^0 J}=\tilde J \cap {}^0G$ (the unique maximal compact subgroup of $\tilde J$) and fix a representation ${}^0 \rho \in \Irr({{}^0 J})$ that is contained in $\tilde \rho$ and contains $\rho$. 
We will also consider the pair $({}^\sharp J:=\cap_{\chi \in X_{{}^0 J}({}^0 \rho)} \ker (\chi), {}^\sharp \rho)$, this is the pair $({}^\dag H,{}^\dag \sigma)$ attached to $(\tilde H,\tilde \sigma) =({}^0 J,{}^0 \rho)$ in Lemma \ref{clifbis}. 
Let ${}^\natural J :={}^\dag({}^\flat G) \cap {}^0 J$. For any $ \nu\in X({}^\natural J/J)$, we denote an extension of it to $X(G/{}^\flat G)$ as $\bar \nu$ and let $\tilde \nu = \Res_{\tilde J}^G(\bar \nu)$ and ${}^0 \nu = \Res_{{}^0 J}^G(\bar \nu)$. 

We claim that  ${}^\natural J \subset {}^\sharp J$. 
To see this, let $x \in {}^\natural J$. Then $\bar \chi(x) = 1$ for all $\bar \chi \in X_{G}(\pi)$. Let ${}^0 \nu \in X_{{}^0 J }({}^0 \rho )$ and $\nu =\Res_{{}^\natural J}^{{}^0 J} {}^0 \nu$.  Then for some extensions $\bar \nu$ (and $\tilde\nu$) of $\nu$, we have $\tilde \nu  \in X_{\tilde J }(\tilde\rho )$ and $\bar\nu\in X_{G}(\pi)$. In particular, for every ${}^0 \nu \in X_{{}^0 J }({}^0 \rho )$, $\bar \nu(x) = 1$.  This implies that ${}^0\nu(x) = 1$ for every ${}^0 \nu \in X_{{}^0 J }({}^0 \rho )$ which  implies that $x \in {}^\sharp J$.

Now consider the pair $({}^\natural J, {}^\natural \rho:=\Res_{{}^\natural J}^{{}^\sharp J}({}^\sharp \rho))$. Since $\Res_{J}^{{}^\sharp J}({}^\sharp \rho) = \rho$, we see that ${}^\natural \rho$ is irreducible. We adapt the case ${}^\flat G={}^0 G$ treated in \cite[Proposition 5.4]{BK98} in the following way:

\begin{theorem}\label{genpro54}
Let $(J,\rho)\in \CJ_G$ and $(\tJ,\tilde \rho)$ as in Definition \ref{pairsCJ}. 
\begin{enumerate}
\item For any $\nu \in X({\,{}^\natural J}/J)$ the pair $({{}^0 J},{}^0 \rho \otimes {}^0 \nu)$ is a $[G,\pi\otimes \bar\nu]_G$-type. 
\item For any $\nu \in X({\,{}^\natural J}/J)$ we have
\[({{}^0 J},{}^0 \rho \otimes {}^0 \nu) \approx_G ({\,{}^\natural J},{}^\natural \rho\otimes \nu),\]
that is, the pair $({\,{}^\natural J},{}^\natural \rho\otimes \nu) $ is also a $[G,\pi\otimes \bar\nu]_G$-type. 
\item 
The pair $(J,\rho)$ is a type in $G$ such that 
{\[{\fkS(J,\rho)}=\bigsqcup_{\,{} \nu \in X({\,{}^\natural J}/J)}{\fkS({{}^0 J},{}^0 \rho\otimes {}^0 \nu)}= \{[G,\pi\otimes \bar\nu]_G \colon \, \nu \in X({\,{}^\natural J}/J)\} ,\]}
\end{enumerate}
\end{theorem}
\begin{proof}
(1) Follows from \cite[Proposition 5.4]{BK98}. 

(2)
Let $ \nu \in X({}^\natural J/J)$ and fix an extension $\bar \chi \in X(G/{}^\dag({}^\flat G))$ for each $\chi\in  X({\,{}^\sharp J}/{}^\natural J) $. 
By definition of ${}^\dag({}^\flat G)$ we have 
$$\pi  \otimes \bar \nu \simeq \pi \otimes \bar \nu \otimes \bar \chi, \quad \forall \chi \in X({\,{}^\sharp J}/{}^\natural J).$$
This implies that $[G,\pi\otimes  \bar \nu]_G=[G,\pi\otimes  \bar \nu \otimes \bar \chi]_G$ for any $\chi \in X({\,{}^\sharp J}/{}^\natural J)$. Using (1), this implies that $({}^0 J, {}^0 \rho \otimes {}^0 \nu \otimes {}^0 \chi)$ is also a $[G,\pi\otimes  \bar \nu]_G$-type for each $ \chi \in X({}^\sharp J/ {}^\natural J)$.

Using Frobenius reciprocity, Lemma \ref{GenClif} and Lemma \ref{clifbis} we have any $\sigma \in \Irr(G)$:
\begin{align*}
\Hom_{{}^\natural J}({}^\natural \rho\otimes \nu,\Res_{{}^\natural J}^G(\sigma))
 &\simeq  \Hom_{{\,{}^\sharp J}}(\ind_{{}^\natural J}^{{\,{}^\sharp J}}({}^\natural \rho\otimes  \nu),\Res_{{\,{}^\sharp J}}^G(\sigma))\\
&\simeq  \Hom_{{\,{}^\sharp J}}({}^\sharp \rho\otimes\nu \otimes \BC[{\,{}^\sharp J}/{}^ \natural J],\Res_{{\,{}^\sharp J}}^G(\sigma)) \\
&\simeq  \bigoplus_{\chi \in X({\,{}^\sharp J}/{}^\natural J)}\Hom_{{\,{}^\sharp J}}({}^\sharp \rho\otimes \nu \otimes \chi,\Res_{{\,{}^\sharp J}}^G(\sigma))\\
&\simeq 
\bigoplus_{\chi \in X({\,{}^\sharp J}/{}^\natural J)}m_J({}^0 \rho)\Hom_{{\,{}^0 J}}({}^0 \rho\otimes {}^0 \nu \otimes {}^0 \chi,\Res_{{\,{}^0 J}}^G(\sigma))
\end{align*}
In the last isomorphism above, we have used that $\ind_{{}^\sharp J}^{{}^0 J} {}^\sharp \rho\otimes \nu \otimes \chi \cong m_J({}^0 \rho) \cdot  ({}^0\rho \otimes {}^0\nu \otimes {}^0 \chi)$ and Frobenius reciprocity. 
Therefore, $ \Hom_{{}^\natural J}({}^\natural \rho\otimes \nu,\Res_{{}^\natural J}^G(\sigma))\neq 0$  if and only if $\sigma$ belongs to $\fkR^{[G,\pi \otimes \bar \nu \otimes \bar \chi]}$ for some $\chi \in X({}^\sharp J/{}^\natural J)$ but this latter is just $\fkR^{[G,\pi \otimes \bar \nu]}$.

(3) Using Frobenius reciprocity again, we have
\begin{align*}
\Hom_J(\rho,\Res_J^G(\sigma)) &\simeq  \Hom_{{\,{}^\natural J}}(\ind_J^{{\,{}^\natural J}}(\rho),\Res_{{\,{}^\natural J}}^G(\sigma))\\
&\simeq  \bigoplus_{\nu \in X({\,{}^\natural J}/J)}\Hom_{{\,{}^\natural J}}({}^\natural \rho\otimes \nu,\Res_{{\,{}^\natural J}}^G(\sigma)).
\end{align*}
This shows that 
\[{\fkS(J,\rho)}= \bigcup_{ \,\nu \in X({\,{}^\natural J}/J)}{\fkS({{}^\natural J},{}^\natural \rho\otimes \nu)}  = \bigcup_{ \,\nu \in X({\,{}^\natural J}/J)}{\fkS({{}^0 J},{}^0 \rho\otimes {}^0 \nu)}.\] 
Note that for any $\nu,\nu' \in X({\,{}^\natural J}/J)$ we have
\[ [G,\pi\otimes \bar\nu]_G = [G,\pi\otimes \bar\nu']_G \implies \Res_{{}^\dag ({}^\flat G) \cap {}^0 G}^G(\bar\nu^{-1} \bar\nu')  \text{ is trivial }\implies  \nu= \nu'.\]
This shows 
\[{\fkS(J,\rho)}= \bigsqcup_{ \,\nu \in X({\,{}^\natural J}/J)}{\fkS({{}^0 J},{}^0 \rho\otimes {}^0 \nu)}.\] 
This concludes the proof of (3). 
\end{proof}
\begin{corollary}\label{indquotients}  
 Let $(J,\rho)\in \CJ_G$.  
Any irreducible sub-quotient of $\ind_J^G(\rho)$ is isomorphic to $\pi \otimes \chi$ for some $\chi \in X(G/{}^\flat G)$. 

\end{corollary}
\begin{proof}
We have already seen in Theorem \ref{genpro54} that $(J,\rho)$ is a type. Since $ \ind_J^G(\rho)$ is generated by its $\rho$-isotypic subspace, $ \ind_J^G(\rho)\in \fkR_\rho$. Hence, for any irreducible sub-quotient  $\sigma$ of $\ind_J^G(\rho)$ we have $\Hom_G( \ind_J^G(\rho), \sigma)\neq 0$, that is $\fkI(\sigma) \in {\fkS(J,\rho)} $. 
Accordingly, we have $\fkI(\sigma) =[G,\pi\otimes \bar{\nu}]_G$ for some $ \nu \in X({\,{}^\natural J}/J)$, and so there exists $g\in G$ such that $\sigma= \pi^g\otimes  \bar{\nu}\chi\simeq \pi \otimes  \bar{\nu}\chi$ for some $\chi \in X(G/{}^0 G)$.
\end{proof}

\begin{remark}\label{tjsubig}
Let $ (J , \rho) \in \CJ_G$. 
By Schur's lemma 
we have an equality $\mathcal{I}_G(\rho)\cap \tJ ={\Int}_{\tJ}(\rho)$. 
Hence, since $\tilde \rho \simeq \Ind_{{\Int}_{\tJ}(\rho) }^{\tJ}(\widehat{\rho})$ where $\widehat{\rho} $ is the irreducible representation defined in (4) Lemma  \ref{GenClif}, we may replace the pair $(\tJ,\tilde \rho)$ by $({\Int}_{\tJ}(\rho),\widehat{\rho})$. 
Accordingly, for any pair $ (J , \rho) \in \CJ_G$ we may assume that $\tJ \subset \CI_G(\rho)$. 
\end{remark}


\begin{lemma}\label{IGtilde}
Let $\tilde J \in \Omega(G/Z)$, $J=\tilde J \cap {}^\flat G$ and $\tilde \rho \in {\Irr}(\tilde J)$.  
For any $\rho\in \CO_J(\tilde \rho)$, we have $\CI_G(\widetilde{\rho})\subset \widetilde{J}\CI_G({\rho})\widetilde{J}$. 
In particular, 
\[\mathcal{I}_G(\rho) \subset \tilde J \Rightarrow\CI_G(\tilde\rho) = \tilde J\text{ i.e. }(J,\rho)\in \CJ_G.\]
\end{lemma}
\begin{proof}
If $g \in G$ satisfies $ \Hom_{\tilde J \cap \tilde J^g}(\tilde \rho, \tilde \rho^g)\neq 0$, then $ \Hom_{ J \cap  J^g}(\Res_J^{\tilde J}(\tilde \rho), \Res_J^{\tilde J} (\tilde \rho)^g)\neq 0.$ 
Using Mackey and Clifford theories we deduce that there exists $h,h'\in \tilde J$ such that $h'gh^{-1} \in \mathcal{I}_G(\rho)$, which shows the desired equality. 
\end{proof}
\begin{definition}\label{CWCT} A pair $(J,\rho)\in \CJ_G$ (and its class $[J,\rho]_G$) will be called \textit{weakly cuspidal} if 
(i) $ \mathcal{I}_G(\rho) \subset \tJ$, 
and \textit{cuspidal} if in addition 
(ii) $ m_J(\tilde\rho)=1$. 
Write $\CJ_G^{wc}$ for the set of weakly cuspidal types. 
\end{definition}

\begin{remark}\label{tjeqig}
In view of Remark \ref{tjsubig} we may assume that if $(J, \rho)$ is weakly cuspidal, then $\tJ= \mathcal{I}_G(\rho)$ and $\Res_J^{\tJ}(\tilde \rho)=m_J(\tilde \rho) \cdot \rho$. 
\end{remark}

In the case ${}^\flat G  = {}^0 G$, it is shown in \cite[Proposition 5.6]{BK98} that the Hecke algebra $\CH(G, \rho)$ is commutative for cuspidal types $(J, \rho)$. As noted in \cite[\S 5.5]{BK98}, (i) and (ii) are satisfied when $G = {}^0G$. 
The existence of a type that satisfies (i) and (ii) is known when $G = \GL_n$ or its inner forms and when $G$ is a classical group provided the residue characteristic of $F$ is odd (see \cite{BK93, MS14}).

\subsection{Isomorphism of Hecke algebras} 
\begin{lemma}\label{heckeisintertwining}
Let $(J,\rho)\in \CJ_G^{wc}$.  
We have an isomorphism of algebras $\CH(G,\rho) = \CH(\mathcal{I}_G(\rho) , \rho) $. 
\end{lemma}
\begin{proof}
The proof given here is essentially the same as the one in \cite[Proposition 5.6]{BK98}. We write the details for completeness. 

Let $f \in \CH(\CI_G(\rho), \rho)$; we view elements on the Hecke algebra as functions as in \cite[Section 2.1]{BK98}. Define $\tilde f $ on $\CH(G, \rho)$ by setting 
$\tilde f(x) = 0$ if $x \notin \CI_G(\rho)$. The map $\Phi: \CH(\CI_G(\rho), \rho) \rightarrow \CH(G, \rho),\; f \rightarrow \tilde f$ is an algebra embedding. It remains to see that $\Phi$ is surjective. To prove this, it suffices to show that for $h\in \CH(G, \rho)$, the support of $h$ is contained in $\CI_G(\rho)$. However, $g \in G$ lies in the support of a function in $\CH(G, \rho)$ if and only if $g$ intertwines $\rho$ (see \cite[Section 2.1]{BK98}). This finishes the proof of the lemma. 
\end{proof}


\begin{lemma}\label{freeoverz}
Let $(J,\rho)\in \CJ_G^{wc}$. 
The Hecke algebra $\CH(G,\rho)$ is a free $\BC[{}^\dag J/J]$-module of rank $m_J(\tilde\rho)^2$. 
In particular, if $m_J(\widetilde{\rho})=1$ then $\CH(G,\rho) \simeq \BC[{}^\dag J/J]$ is commutative.
\end{lemma}
\begin{proof}

We have the following isomorphism of $\BC$-modules
\begin{align} 
\CH(G,\rho) &=\End_{{\CI_G(\rho)}}(\ind_J^{{\CI_G(\rho)}}(\rho)) \label{eq32}\\
&\simeq\Hom_{{}^\dag J}\left({}^\dag\rho\otimes \BC[{}^\dag J/J],\Res_{{}^\dag J}^{{\CI_G(\rho)}}\circ\,\ind_{ J}^{{\CI_G(\rho)}}( \rho)\right)\\
&\simeq \bigoplus_{j\in {\CI_G(\rho)}/{}^\dag J} \Hom_{{}^\dag J}\left({}^\dag\rho\otimes \BC[{}^\dag J/J],\ind_J^{{}^\dag J}(\rho^j)\right) \\
&\simeq m_J(\tilde\rho)^2 \Hom_{{}^\dag J}({}^\dag \rho\otimes \BC[{}^\dag J/J],{}^\dag\rho\otimes \BC[{}^\dag J/J]) 
\end{align}

The first isomorphism follows from Lemma \ref{heckeisintertwining}, the second and fourth from Lemma \ref{clifbis} and the third from Mackey formula. 


The map $\mu\colon\BC[{}^\dag J/J] \hookrightarrow  \End_{{}^\dag J}({}^\dag \rho \otimes_\BC \BC[{}^\dag J/J])$, that sends an element $\bar w$ to the endomorphism 
\[v \otimes \bar j \mapsto   v \otimes \bar j \bar w^{-1}, \quad \forall v \in {}^\dag \rho , \forall \bar j \in {}^\dag J/J,\]
yields an embedding of $\BC$-algebras 
\[\BC[{}^\dag J/J] \hookrightarrow  \End_{{}^\dag J}({}^\dag \rho \otimes_\BC \BC[{}^\dag J/J])).\] 
But since ${}^\dag \rho$ is irreducible and $J$ acts trivially on $\BC[{}^\dag J/J]$ this embedding is actually an isomorphism 
\[\End_{{}^\dag J}({}^\dag \rho \otimes_\BC \BC[{}^\dag J/J]))\simeq \BC[{}^\dag J/J] .\]
This shows that $\mu\colon \BC[{}^\dag J/J] \iso  \End_{{}^\dag J}({}^\dag \rho \otimes_\BC \BC[{}^\dag J/J])$ is an isomorphism.
The algebra $\BC[{}^\dag J/J]$ acts on $\End_{{\CI_G(\rho)}}(\ind_J^{{\CI_G(\rho)}}(\rho))$ as follows: 
\[w\cdot   \phi = \phi \circ \ind_{{}^\dag J}^{\CI_G(\rho)} (\mu(w)), \quad\forall w \in \BC[{}^\dag J/J], \, \forall \phi \in  \End_{{\CI_G(\rho)}}(\ind_J^{{\CI_G(\rho)}}(\rho)).\] 
Moreover, the isomorphisms in the equations below 
 \eqref{eq32} are all $\BC[{}^\dag J/J]$-equivariant. 
This proves the lemma.
\end{proof}
\begin{remark}
    When ${}^\flat G = {}^0 G$, the analogue of the result above for the Hecke algebra $\CH(G,{}^0 \pi)$ is proved in \cite[Proposition 1.6.3.2]{Roc09}. 
\end{remark}
\subsection{Multiplicites for types} 
 Let $(J,\rho) \in \CJ_G^{wc}$. 
In this subsection, we prove that $m_{{}^\flat G}(\pi) = m_{J}(\tilde\rho)$. 
We will deduce several consequences of this result in the subsequent subsections.
\subsubsection{}\label{322}
We recall \cite[\S 8.3]{Vigneras1996}: For any $K\in \Omega(G/Z)$ and any $\sigma \in {\Irr}(K) $, the following statements are equivalent
\begin{enumerate}[(i)]
\item $\ind_{K}^G (\sigma)$ is irreducible,
\item $\End_G (\ind_{K}^G (\sigma))=\BC$,
\item $\CI_G(\sigma)=K$,
\item $\sigma$ is not contained in $\ind_{K \cap K^g}^{K} \Res_{K \cap K^g}^{K^g} (\sigma^g)$ for any $g \not\in K$.
\end{enumerate} 

Consequently  
$ {}^\flat\tilde \pi:=\ind_{\tilde J}^{{}^\flat G \tilde J} (\tilde\rho)\in {\Irr}({}^\flat G \tilde J)
 \text{ and }{}^\flat \pi:=\ind_{ J}^{{}^\flat G } (\rho)\in {\Irr}({}^\flat G).$ 
For any $g \in G$, we have  
\begin{align*}
\Hom_{{}^\flat G}({{}^\flat\pi},{{}^\flat \pi}^g)&\simeq \Hom_{{}^\flat G}(\ind_{J}^{{}^\flat G}({\rho}),\ind_{J^g}^{{}^\flat G}({\rho}^g))\\
&\simeq \bigoplus_{h \in J^g\backslash {}^\flat G/J} \Hom_{J \cap J^{hg}}(\rho, \rho^{hg}).
\end{align*}
So since ${}^\flat \pi$ is irreducible the left Hom space is non zero if and only if $g \in {}^\flat G\mathcal{I}_G(\rho)$. 
Accordingly, 
\begin{equation}\label{intpi} {\Int}_G({{}^\flat \pi}^g)=\mathcal{I}_G({{}^\flat \pi}^g)={}^\flat G {\mathcal{I}_G(\rho)}^g, \quad \forall g\in G.
\end{equation}
\subsubsection{}
\begin{theorem}\label{typesmultiplicites}
We have
\[m_{{}^\flat G}(\pi)= m_{{}^\flat G}({}^\flat\tilde  \pi)  = m_{J}(\tilde \rho).\]
\end{theorem}
\begin{proof}
{

Using Mackey theory we can easily see that ${}^\flat \tilde\pi$ contains $ {}^\flat\pi$ and that $\pi$ contains ${}^\flat \tilde\pi$, so given \eqref{intpi} and using Lemma \ref{GenClif} we have
$$\Res_{{}^\flat G}^{{}^\flat G \tilde J}({}^\flat \tilde\pi)= m_{{}^\flat G }({}^\flat \tilde\pi)\bigoplus_{h \in \tilde J/\CI_{G}(\rho)} {{}^\flat\pi}^h \text { and }\Res_{{}^\flat G \tilde J}^G(\pi)= m_{{}^\flat G \tilde J}(\pi)\bigoplus_{h \in G/{}^\flat G \tilde{J}} {{}^\flat \tilde\pi^h}.$$
Similarly $$\Res_{{}^\flat G }^G(\pi)= m_{{}^\flat G }(\pi)\bigoplus_{h \in G/{}^\flat G \CI_{G}(\rho)} {{}^\flat\pi}^h.$$ 

We are going to compute the dimension of $\Phi:=\Hom_{G}(\ind_{{J}}^{G}({\rho}),\ind_{\widetilde{J}}^G(\widetilde{\rho})) $ in various ways, mainly by playing with Mackey Theory and Frobenius reciprocities. 

\begin{itemize}
\item First, $
\Phi\simeq\Hom_{{}^\flat G}({}^\flat\pi,\Res^G_{{}^\flat G}(\pi))$. 
So by the irreducibility of ${}^\flat\pi$, 
and Clifford theory 
\begin{equation}\dim_\BC(\Phi)=\bigoplus_{G/{\Int}_G({}^\flat \pi)} m_{{}^\flat G}(\pi) \dim_\BC(\Hom_{{}^\flat G}({}^\flat\pi,({}^\flat\pi)^h)= m_{{}^\flat G}(\pi).\end{equation}
\item Recall that $G/{}^\flat G\tilde J$ is finite, so $\ind_{{}^\flat G\tilde J}^G( {}^\flat \tilde\pi)= \Ind_{{}^\flat G\tilde J}^G( {}^\flat \tilde\pi)$ is admissible.  
Using Frobenius reciprocity and Mackey theory
\begin{align*}
\Phi		&\simeq\Hom_{{}^\flat G\tilde J}(\Res^G_{{}^\flat G \tilde J}(\ind_{J}^G({\rho})), {}^\flat \tilde\pi)\\
 		& \simeq \bigoplus_{g \in J \backslash G/{}^\flat G \tilde J} \Hom_{{}^\flat G\tilde J}(\ind_{{}^\flat G\tilde J \cap J^g}^{{}^\flat G\tilde J}\Res_{J^g}^{{}^\flat G\tilde J \cap  J^g}({\rho}^g), {}^\flat \tilde\pi)\\
		& \simeq \bigoplus_{g \in  G/{}^\flat G \tilde J} \Hom_{{}^\flat G\tilde J}(\ind_{J^g}^{{}^\flat G\tilde J}({\rho}^g), {}^\flat \tilde\pi)
		\end{align*}
Observe that $\ind_{J^g}^{{}^\flat G\tilde J}({\rho}^g) =\ind_{{}^\flat G}^{{}^\flat G\tilde J} \ind_{J^g}^{{}^\flat G}({\rho}^g) =\ind_{{}^\flat G}^{{}^\flat G\tilde J}({}^\flat \pi^g)$. So
		\begin{align*}
\Phi    & \simeq \bigoplus_{g \in G/{}^\flat G \tilde J} \Hom_{{}^\flat G\tilde J}(\ind_{{}^\flat G}^{{}^\flat G\tilde J}({{}^\flat\pi}^g), {}^\flat \tilde\pi)\\
        & \simeq \bigoplus_{g \in G/{}^\flat G \tilde J} \Hom_{{}^\flat G}({{}^\flat\pi}^g, \Res_{{}^\flat G}^{{}^\flat G\tilde J}({}^\flat \tilde\pi))\\
        & \simeq m_{{}^\flat G}({}^\flat \tilde\pi)\bigoplus_{j \in {}^\flat G \tilde J/{\Int}_{{}^\flat G \tilde J}( {}^\flat\pi)} \, \bigoplus_{g \in G/{}^\flat G \tilde J}   \Hom_{{}^\flat G}({{}^\flat\pi}^g, {{}^\flat\pi}^j)
\end{align*} 
For any $j\in \tilde J/\CI_{G}( \rho)$, The last Hom space is nonzero only if $g^{-1}j\in \mathcal{I}_G({{}^\flat\pi})={}^\flat G {\mathcal{I}_G(\rho)} $, hence $g \in {}^\flat G \tilde{J} \subset {\Int}_G({}^\flat \tilde\pi^g)$. 
So,
$$\Phi\simeq m_{{}^\flat G}({}^\flat \tilde\pi)\bigoplus_{j \in \tilde J/\mathcal{I}_{G}( \rho)}  \Hom_{{}^\flat G}({{}^\flat\pi}, {{}^\flat\pi}^j)=m_{{}^\flat G}({}^\flat \tilde\pi) \End_{{}^\flat G}({{}^\flat\pi} )$$
Therefore, 
\begin{equation}\dim_\BC(\Phi)=m_{{}^\flat G}({}^\flat \tilde\pi).\end{equation}
\item Finally, 
\begin{align*}
\Phi	 & \simeq \bigoplus_{g \in G/{}^\flat G     \tilde J} \Hom_{{}^\flat G}({{}^\flat\pi}^g, \Res_{{}^\flat G}^{{}^\flat G\tilde J}({}^\flat \tilde\pi))\\	
		&\simeq \bigoplus_{g \in G/{}^\flat G     \tilde J} \bigoplus_{g\in \tilde J\backslash {}^\flat G \tilde J /{}^\flat G}  \Hom_{{}^\flat G}({{}^\flat\pi}^g, \ind_{{}^\flat G\cap \tilde J^h}^{{}^\flat G} \Res_{{}^\flat G\cap \tilde J^h}^{\tilde J^h}(\tilde \rho^h) )\\
		&\simeq \bigoplus_{g \in G/{}^\flat G     \tilde J}   \Hom_{{}^\flat G}({{}^\flat\pi}^g, \ind_{J}^{{}^\flat G} \Res_{J}^{\tilde J}(\tilde \rho) )\\
		&\simeq m_{J}(\tilde \rho) \bigoplus_{g \in G/{}^\flat G     \tilde J}  \bigoplus_{j \in \tilde J/\mathcal{I}_G(\rho)} \Hom_{{}^\flat G}({{}^\flat\pi}^g, {{}^\flat\pi}^j )
\end{align*}
Accordingly $\Phi \simeq m_{J}(\tilde \rho) \End_{{}^\flat G}({{}^\flat\pi} )$.  
In conclusion $\dim_\BC(\Phi)= m_{J}(\tilde \rho)$. 
\end{itemize}}
\end{proof}

\begin{corollary}\label{cor42}
We have
$$\Res_{{}^\flat G}(\pi)=m_J(\tilde \rho)\bigoplus_{h\in G/{}^\flat G \mathcal{I}_G(\rho)} {{}^\flat\pi}^h.$$
\end{corollary}

\begin{lemma}\label{jdag} Let ${}^\dag J$ be as in Lemma \ref{clifbis} where $\tH = \tJ$ and $\tilde\sigma = \tilde\rho$. Then
\[\bigcap_{\chi \in X_{G}(\pi)} \ker(\chi)= {}^\dag J\; {}^\flat G \]
In particular, ${}^\dag J = \CI_G(\rho) \cap \bigcap_{\chi \in X_{G}(\pi)} \ker(\chi)$ and ${}^\dag ({}^\flat\pi)=\ind_{{}^\dag J}^{{}^\dag J\; {}^\flat G }({}^\dag \rho)$.
\end{lemma}

\begin{proof}

Consider the inclusion $\psi \colon \CI_G(\rho)/J \hookrightarrow G/{}^\flat G$. 
The subgroup $\psi(\CI_G(\rho)/J )$ is a finite index subgroup of $ G/{}^\flat G$. 
Hence, any character of $\CI_G(\rho)/J $ can be extended to a character of $G/{}^\flat G$ in $[ G/{}^\flat G : \psi(\CI_G(\rho)/J ) ]$ ways. 
Let \[\psi^*\colon X(G/{}^\flat G)\to X(\tilde J/J)\] be the map induced by the restriction.

Note that for any $\nu \in X_{\tilde J}(\tilde \rho)$ if $\tilde{\rho}\simeq  \tilde{\rho} \otimes \nu$ then $\pi \simeq \pi \otimes \bar\nu$ where $\bar \nu$ is any extension of $\nu$ to $G$. 
This shows that $(\psi^{*})^{-1}(X_{\tilde J}(\tilde \rho))\subset X_{G}(\pi)$ where $(\psi^{*})^{-1}(X_{\tilde J}(\tilde \rho))$ represents the set of all extensions of elements in $X_{\tilde J}(\tilde \rho)$ to $G$. 
Accordingly
\[\bigcap_{\chi \in X_{G}(\pi)} \ker(\chi) \subset \bigcap_{\bar\nu \in (\psi^{*})^{-1}(X_{\tilde J}(\tilde \rho))
} \ker(\nu)={}^\flat G \bigcap_{\nu \in X_{\tilde J}(\tilde \rho)
} \ker(\nu)={}^\dag J\; {}^\flat G .\]
Using \eqref{intpi} and Theorem \ref{typesmultiplicites} we have
    \begin{align*}
     [\CI_G(\rho) {}^\flat G:\bigcap_{\nu \in X_{G}(\pi)} \ker(\nu) ]
    &=[\text{Int}_G({}^\dag \pi):\bigcap_{\nu \in X_{G}(\pi)} \ker(\nu) ]\\
    &=(m_G(\pi))^2=(m_J(\tilde \rho))^2
    =[\CI_G(\rho):{}^\dag J]\\
    &=[\CI_G(\rho) {}^\flat G : {}^\dag J\; {}^\flat G ].
    \end{align*}  
Therefore, $ \bigcap_{\chi \in X_{G}(\pi)} \ker(\chi) $ and  ${}^\flat G{}^\dag J $ must be equal and also ${}^\dag J = \CI_G(\rho)\; \cap \;\bigcap_{\chi \in X_{G}(\pi)} \ker(\chi)$. 

For the last statement observe that $\ind_{{}^\dag J}^{{}^\dag J\; {}^\flat G }({}^\dag \rho)$ is irreducible (since $\CI_G({}^\dag\rho) \subset \CI_G(\rho)$), 
occurs in $\Res_{{}^\dag J\; {}^\flat G }^{G}(\pi)$ and satisfies $\Res_{{}^\flat G}^{{}^\dag J {}^\flat G}(\ind_{{}^\dag J}^{{}^\dag J\; {}^\flat G }({}^\dag \rho))= \ind_{J}^{{}^\flat G}(\rho)={}^\flat \pi$. 
We now conclude using (4) Lemma \ref{clifbis}. 
\end{proof}

\subsection{Center of Hecke algebras}

The following result describes the center of the Hecke algebra of a supercuspidal type. 
\begin{theorem}\label{centerhecke} 
Let $(J,\rho)\in \CJ_G^{wc}$. We have the following isomorphisms of $\BC$-algebras
$$\mathcal{Z}(\CH(G,\rho))= \mathcal{Z}(\CH(\CI_G(\rho),\rho))= \CH({}^\dag J, \rho) \simeq \BC[{}^\dag J/J].$$
The two first are canonical, while the last is not.
\end{theorem}
\begin{proof}
The first equality follows readily from Lemma \ref{heckeisintertwining}. 
Note that we have $\CH(G,\rho) \simeq \CH(G,{}^\flat \pi)$ \cite[Chapitre I \S 8.6 (b)]{Vigneras1996}.
Given Lemma \ref{clifbis}, the proof of \cite[Proposition 1.6.3.2]{Roc09} shows (upon replacing ${}^0G$ in \emph{loc. cit.} by ${}^\flat G$, 
which amounts to replacing ${}^0 \pi$ in \emph{loc. cit.} by ${}^\flat \pi$) that 
\[\mathcal{Z}(\CH(G,{}^\flat\pi))= \CH( {}^\dag({}^\flat G) ,{}^\flat\pi),\]
where ${}^\dag({}^\flat G)=\bigcap_{\chi \in X_{G}(\pi)} \ker(\chi)={}^\dag J\; {}^\flat G $. 
Lemma \ref{heckeisintertwining} applied to ${}^\dag({}^\flat G)$  shows that 
\[\CH( {}^\dag({}^\flat G) ,{}^\flat \pi)= \CH( \CI_{{}^\dag({}^\flat G)}(\rho) ,\rho)= \CH( {}^\dag({}^\flat G) \cap \CI_G(\rho) ,\rho).\]
Now applying Lemma \ref{jdag} we get ${}^\dag({}^\flat G) \cap \CI_G(\rho)={}^\dag J$ and so $\mathcal{Z}(\CH(G,{}^\flat\pi))=\CH( {}^\dag J ,\rho)$. 

We could have also reproduced the same argument of \cite[Proposition 1.6.3.2]{Roc09} with ${}^\dag J$ playing the role
of ${}^\dag({}^\flat G)$  and $\CI_G(\rho)$ that of $G$ and prove directly the second isomorphism above. 

For the last isomorphism, we have the isomorphism 
\[\mu \colon\BC[{}^\dag J/J] \iso \End_{{}^\dag J}({}^\dag \rho\otimes \BC[{}^\dag J/J])=\CH( {}^\dag J ,\rho)\] 
defined in the proof of Lemma \ref{freeoverz}. This concludes the proof of the theorem.
\end{proof}
\subsection{Criterion for the Hecke algebra of a supercuspidal type to be commutative}

\begin{proposition}\label{Prop:HAComm}
Let $(J,\rho)\in \CJ_G^{wc}$. 
The following statements are equivalent:
\begin{enumerate}
\item The representation $\Res_{{}^\flat G}^G(\pi)$ is multiplicity free. 
\item The representation $\Res_J^{\widetilde{J}}(\tilde\rho)$ is also multiplicity free.
\item The Hecke algebra $\CH(G, \rho)$ is commutative. 
\end{enumerate}
\end{proposition}
\begin{proof}

Given that ${}^\flat G$ is open we know that $\CH(G, {}^\flat\pi) \simeq \CH(G, \rho)\simeq \CH(\mathcal{I}_G(\rho), \rho)$ thanks to the transitivity of the induction. 
So by Theorem \ref{centerhecke} and Lemma \ref{freeoverz} we have the equivalence $(2) \Leftrightarrow (3)$ and Proposition \ref{typesmultiplicites} gives $(1) \Leftrightarrow (2)$.
\end{proof}

\section{Weyl action on (center of) Hecke algebras and a Satake isomorphism}\label{sec:nonsup}
The results in this section are  generalizations of \cite[Section 1.6 - Section 1.9]{Dat99}, where similar results are obtained for cuspidal types in the case ${}^\flat G = {}^0 G$. 
\subsection{$G$-equivalence of types}\label{gequiv}
\subsubsection{} 
For $(J,\rho)\in \CJ_G$, let $[\rho]$ be the set of irreducible representations of $G$ that contain $\rho$ and  $[\pi\otimes \bar\nu]_G$ be the subset of irreducible representations of $G$ whose inertial support contains $\pi\otimes \bar\nu$ for $\nu \in X({\,{}^\natural J} /J)$.

\begin{definition}\label{def:wct} We say that two types $(J,\rho)$ and $(J',\rho')$ are $G$-equivalent and write it as $(J,\rho) \cong_G (J',\rho')$ if $\ind_{J}^G(\rho) \simeq \ind_{J'}^G (\rho').$
\end{definition}

\begin{lemma}\label{NPT} 
Let $(J,\rho)$ and $(J', \rho')$ be two types in $\CJ_G$. The following properties are equivalent:
\begin{enumerate}
    \item $[\rho] \cap [\rho'] \neq \emptyset$;
    \item[(1)'] $(J,\rho) \approx_G (J', \rho')$;
    \item $[\rho] = [\rho']$;
    \item $\Hom_G(\ind_{J'}^G (\rho'), \ind_J^G(\rho)) \neq 0$.\\ 
    If $(J,\rho), (J', \rho')\in \CJ_G^{wc}$, this is also equivalent to
    \item $(J,\rho) \cong_G (J',\rho')$.
\end{enumerate}
\end{lemma}
\begin{proof}
$(1)\Leftrightarrow (2)$ is clear since any two orbits in ${\Irr}(G)$ under the action of $X(G/{{}^\flat G})$ are disjoint or equal and by Theorem \ref{genpro54} we have 
\[[\rho] =\bigsqcup_{ \nu \in X({\,{}^\natural J}/J)}[\pi\otimes\bar\nu]_G \text{ and } [\rho'] =\bigsqcup_{ \nu \in X({}^\natural J'/J')}[\pi'\otimes \bar\nu]_G.\]

$(1)'\Leftrightarrow (2)$ follows from \cite[Proposition 3.5]{BK98}.

$(2)\Leftrightarrow (3)$. Let $I$ be a system of representatives of ${}^\dagger J \backslash G / J'$. 
Using Mackey formula and Frobenius reciprocity we have 
\begin{align*}
    \Hom_G(\ind_{J'}^G (\rho'), \ind_{J}^G (\rho)) 	&\simeq \Hom_G(\ind_{J'}^G (\rho'), \ind_{J}^G (\Res_J^{{}^\dag J }({}^\dag \rho))\\
    	&\simeq \Hom_G(\ind_{J'}^G (\rho'),\ind_{{}^\dag J }^G({}^\dag \rho \otimes_{\BC} \BC[{}^\dag J/J]) )\\
    	& \simeq \bigoplus_{x \in I}  \Hom_{J' \cap{}^\dag J^{x}}(\rho',({}^\dag \rho \otimes_{\BC} \BC[{}^\dag J /J])^{x} )\\
     	& \simeq \bigoplus_{x \in I}  \Hom_{J' \cap{}^\dag J^{x}}\left(\rho',{}^\dag \rho^{x} \otimes_{\BC} \BC[({}^\dag J )^{x}/J^{x}] \right)\\
     & \simeq \bigoplus_{x \in I}  \Hom_{J' \cap J^{{x}}}(\rho',{}^\dag \rho^{x}) \otimes_{\BC} \BC[({}^\dag J )^{x}/J^{x}] \\
     &\simeq \bigoplus_{x \in I}  \Hom_{J' \cap J^{{x}}}(\rho',{}^\dag \rho^{x} )\otimes_{\BC} \BC[{}^\dag J^{x}/J^{x} ]\\
     &\simeq \Hom_G(\ind_{J'}^G (\rho'), \ind_{{}^\dag J}^G({}^{\dag} \rho)  )\otimes_{\BC} \BC[{}^\dag J/J].
     \end{align*}
The commutation of the tensor product in the fifth equation comes from the fact that $J'\cap ({}^\dag J)^x = J' \cap J^x $, hence $J'\cap {}^\dag J^x $ acts trivially on $\BC[{}^\dag J^x /J^x]$.
The isomorphism from the sixth to seventh equation follows from Mackey formula and is obtained as follows.  Let $\psi \in \Hom_G(\ind_{J'}^G (\rho'), \ind_{{}^\dag J}^G({}^{\dag} \rho))$ map to $(\psi_i)_{i\in I} \in \left(\bigoplus_{i \in I}  \Hom_{J' \cap{}^\dag J^{x}}(\rho',{}^\dag \rho^{x} )\right)$. The isomorphism is then given by  sending $ \psi \otimes \chi \rightarrow (\psi_i \otimes \chi^{x})_{i \in I}$.
 
Now by Lemma \ref{clifbis} we deduce 
\begin{align} 
	\Hom_G(\ind_{J'}^G (\rho'), \ind_{J}^G (\rho))
			&\simeq m_{{}^\flat G}(\pi) \Hom_G(\ind_{J'}^G (\rho'),\pi  )\otimes \BC[{}^\dag J/J]. \label{decompoindjj}
\end{align}

By Theorem \ref{genpro54}, the pair $(J',\rho')$ is a type for $G$. 
Therefore
\[\Hom_{G}(\ind_{J'}^G(\rho'),\ind_{J}^G(\rho)) \neq 0 \Leftrightarrow \pi \in [\rho'] .\]


$(4) \Rightarrow (1)$ is clear, let us show $(1)' \Rightarrow (4)$ assuming $(J,\rho), (J', \rho')\in \CJ_G^{wc}$: 
By Theorem \ref{genpro54}, we have an equality $[G,\pi']_G= [G,\pi\otimes \bar\nu]_G$ for some $\nu \in X({\,{}^\natural J} /J)$, and this is equivalent to (by Corollary \ref{indquotients}) $\pi \simeq \pi' \otimes \chi$ for some $\chi \in X(G/{}^\flat G)$, which is equivalent to (by (5) Lemma \ref{GenClif}) 
$\ind_{J'}^{{}^\flat G}(\rho') \simeq (\ind_{J}^{{}^\flat G}(\rho))^x \simeq \ind_{J^x}^{{}^\flat G}(\rho^x)$ for some $x\in G$, since (as we saw in \S \ref{322}) {the assumption insures that} $\ind_{J}^{{}^\flat G}(\rho) \in \CO_{{}^\flat G}(\pi)$ and  $\ind_{J'}^{{}^\flat G}(\rho') \in \CO_{{}^\flat G}(\pi')$.
This latter implies $\ind_{J'}^{G}(\rho') \simeq (\ind_{J}^{G}(\rho))^x \simeq \ind_{J}^{G}(\rho)$. 
\end{proof}

\begin{corollary}\label{isogflat0}
 {Let $(J,\rho)\in \CJ_G$. Let $({}^\natural J, {}^\natural \rho \otimes \nu)$  be the type in Theorem \ref{genpro54}}. We have an isomorphism of $\BC$-algebras 
\[\CH(G, \rho) \simeq \bigoplus_{ \nu\in X({\,{}^\natural J}/J)}\CH(G  ,  {}^\natural \rho\otimes \nu).\]
 \end{corollary}
 \begin{proof}
 We first note the isomorphism of $\BC$-modules
\begin{align*}
\End_{G}(\ind_{J}^G(\rho))&=  \bigoplus_{\nu,\nu' \in X({\,{}^\natural J}/J)}\Hom_{G}(\ind_{{\,{}^\natural J}}^G({}^\natural \rho\otimes \nu),\ind_{{\,{}^\natural J}}^G({}^\natural \rho\otimes \nu')).
 \end{align*}
Set ${}^\natural G:= {}^\dag ({}^\flat G)\cap {}^0 G$ (which contains ${}^\flat G$). 
Now, it is immediate to see that for any $\nu \in X({}^\natural J/J)$, the pair $({}^\natural J, {}^\natural \rho \otimes \nu)$ belongs to the set $\CJ_G$ but this time for the case where ${}^\natural G$ is playing the role of ${}^\flat G$. 
 
 Since $[{}^\natural \rho\otimes \nu'] = [{}^\natural \rho\otimes \nu]$ if and only if $\nu=\nu'$ (see Theorem \ref{genpro54}), Lemma \ref{NPT}\footnote{applied to the situation where ${}^\flat G$ is replaced by ${}^\natural G$.} yields 
 an isomorphism of $\BC$-modules 
 \begin{align*}
 \End_{G}(\ind_{J}^G(\rho))&=  \bigoplus_{\nu \in X({\,{}^\natural J}/J)}\End_{G}(\ind_{{\,{}^\natural J}}^G({}^\natural \rho\otimes \nu)).
 \end{align*}
 which is clearly an isomorphism of $\BC$-algebras. 
This concludes the proof. 
\end{proof}

\begin{corollary}\label{cor:2flat-0}
Let $(J, \rho)$ be a cuspidal pair in $\CJ_G$. As in Remark \ref{tjeqig}, we assume $\tilde J = \CI_G(\rho)$ and that $\Res_J^{\tilde J} \tilde\rho = \rho$ is irreducible. Then $({}^\natural J, {}^\natural \rho) = ({}^0 J, {}^0 \rho)$. Moreover,  
\[\CH(G, \rho)\simeq\bigoplus_{\nu\in X( {{}^0 J}/J)} \CH(G  ,  {}^0 \rho\otimes {}^0 \nu).\]
\end{corollary}
\begin{proof}
     Lemma \ref{clifbis} implies that ${}^\dagger J = \tilde J$ and Lemma \ref{jdag} implies that ${}^\dagger({}^\flat G) = \tilde J \; {}^\flat G$. Then ${}^\natural J = {}^0 J \cap\;  {}^\dagger({}^\flat G) = {}^0 J$. It is then clear that ${}^\natural \rho = {}^0\rho$. This proves the first claim. The second claim follows immediately from Corollary \ref{isogflat0}.
\end{proof}

\subsection{Normalizer of a type}\label{sec:nortype}
Let $\bM$ be a Levi subgroup of $\bG$ and ${}^\flat M \subset M \cap {}^\flat G$.  This is an open  subgroup of $M$ of finite index in ${}^0M$ and containing $M^{\text{der}}$. 
We write $\CJ_M$ (resp. $\CJ_M^{wc}$) as in \S \ref{sometypessec}. 
For $(J_M,\rho_M)\in \CJ_M$ we will regularly use the notation $\pi_M:=\ind_{\tilde J_M} ^M (\tilde\rho_M) \in \Irr(M)$. 

The normalizer $N_G(M)$ of a Levi $M$ acts naturally by conjugation on ${\Irr}(M)$.
Recall that $\sigma^n$ denotes the conjugate of any $\sigma\in {\Irr}(M)$ by an element $n\in N_G(M)$ and the pair $(J_M^n,\rho_M^n)$ the conjugate of $(J_M,\rho_M)$. 
Moreover, $\CI_M(\rho_M^n)=(\CI_M(\rho_M))^n$ and $\CH(M,\rho_M)\simeq \CH(M,n(\rho_M))$. 
We are interested in the case where $n$ normalizes $\CJ_M$, this is the case for example when $n$ normalizes ${}^\flat M$. Assume for the rest of section 4 that $N_G(M)$ normalizes ${}^\flat M$.

\begin{proposition}\label{Gequiv}Let $(J, \rho)$ be a $G$-cover (\cite[Definition 8.1]{BK98}) of $(J_M, \rho_M)\in \CJ_M$. 
For any $n \in N_G(M)$, the following statements are equivalent:
\begin{enumerate}
    \item $n[\rho_M] \cap [\rho_M] \neq \emptyset$;
    \item[(1)'] $n[\rho_M] = [\rho_M]$;
    \item $\Hom_M(n(\ind_{J_M}^M(\rho_M)), \ind_{J_M}^M(\rho_M)) \neq 0$;
    \item there exists $m \in M$ such that $mn\in \CI_G(\rho)$. \\
    If $(J_M,\rho_M),\in \CJ_M^{wc}$, this is also equivalent to
    \item $(J_M^n,\rho_M^n)\cong_M (J_M,\rho_M)$.
\end{enumerate}
The group $N_G(\rho_M) := \{ n \in N_G(M)\colon (J_M^n,\rho_M^n)\cong_M (J_M,\rho_M)\}$ is called the normalizer of the type $(J_M,\rho_M)$. 
In particular, Given (1)', the stabilizer of $\fkS(J_M,\rho_M)$ in the Weyl group is $W_{[\rho_M]} :=N_G(\rho_M)/M$.
\end{proposition}
\begin{proof}
As in \cite[Proposition 1.9.1]{Dat99}, observe that $n(\ind_{J_M}^M (\rho_M)) \simeq (\ind_{J_M^n}^M (\rho_M^n))$. 
The equivalences $(1) \Leftrightarrow (1)' \Leftrightarrow (2) \Leftrightarrow (4)$, follow then from Lemma \ref{NPT}.

$(2) \Leftrightarrow (3)$
As we saw in the proof of Lemma \ref{NPT} we have 
\[(2) \Leftrightarrow  \Hom_M(\ind_{J_M}^M(\rho_M), \pi_M^n) \neq 0.\] 
Using Frobenius reciprocity and then Mackey formula we see that the right hand side is equivalent to 
\[\Hom_{J_M\cap \tilde J_M^{m'n}}(\Res_{J_M\cap \tilde J_M^{m'n}}^{ J_M}(\rho_M), \Res_{J_M\cap \tilde J_M^{m'n}}^{ \tilde J_M^{m'n}} (\tilde \rho_M^{m'n}))\neq 0 \quad \text{ for some }m' \in M.\]
Now since $J_M\cap \tilde J_M^{m'n}=J_M\cap J_M^{m'n}$ we deduce (using Clifford theory) that the previous statement is equivalent to 
\[\Hom_{J_M\cap J_M^{mn}}(\Res_{J_M\cap  J_M^{mn}}^{ J_M} (\rho_M), \Res_{J_M\cap J_M^{mn}}^{ J_M^{mn}}( \rho_M^{mn}))\neq 0  \quad \text{ for some }m \in M.\]
Let $P$ be any parabolic subgroup with Levi factor $M$ and a radical unipotent $U$, $\bar U$ its opposite. 
By definition of a cover, we have an Iwahori decomposition for $J$ with respect to any parabolic subgroup with Levi component $M$. 
Now it suffices to observe that 
$$J \cap J^{mn}= (J \cap J^{mn} \cap U ) \cdot (J_M\cap J_M^{mn})\cdot (J \cap J^{mn} \cap \bar U) $$
and that $\rho$ and $\rho^{mn}$ are both trivial on both unipotent factors $J \cap J^{mn} \cap U$ and $J \cap J^{mn} \cap \bar U$. 
Therefore, 
$$\Hom_{J\cap J^{mn}}(\Res_{J\cap  J^{mn}}^{ J} (\rho), \Res_{J\cap J^{mn}}^{J^{mn}}( \rho^{mn})) = \Hom_{J_M\cap J_M^{mn}}(\Res_{J_M\cap  J_M^{mn}}^{ J_M} (\rho_M), \Res_{J_M\cap J_M^{mn}}^{ J_M^{mn}}( \rho_M^{mn})).$$
Which shows $mn\in \CI_G(\rho) \Leftrightarrow (2)$. 
This concludes the proof of the proposition. 
\end{proof}
\begin{corollary}
For any two $G$-covers $(J,\rho)$ and $(J',\rho')$ of two types $(J_M,\rho_M)\in \CJ_M^{wc}$ and $(J_M',\rho_M')\in \CJ_M^{wc}$, the following properties are equivalent:
\begin{enumerate}
\item $[\rho]\cap [\rho'] \neq \emptyset$;
\item $(J,\rho) \cong_G (J',\rho')$.
\end{enumerate}
\end{corollary}
\begin{proof}
Given Proposition \ref{Gequiv} this is the same as \cite[Proposition 4.5.1]{Dat99}. We remark that Proposition 4.5.1 of {\em loc. cit.} assumes Conjecture 1.4 in {\em loc. cit.}, which is verified in Section 1.5 of {\em loc. cit.} in the complex case. See also \cite[Lemma B.3]{BaS20}. 
\end{proof}

\subsection{Weyl action on the center}\label{WeylAcCEnt}
\subsubsection{}{
Let $(J_M,\rho_M)\in \CJ_M^{wc}$.
An element $z\in \mathfrak{Z}^{[\rho_M]}$ is a collection of morphisms $z_\sigma\in \End_M(\sigma), \forall \sigma \in \text{Obj}(\fkR_{\rho_M})$, such that $f\circ z_\sigma=z_\tau \circ f$ for any morphism $f\in \Hom_M( \sigma,\tau), \forall \sigma, \tau \in\fkR_{\rho_M}$. 
In particular, $z_\sigma\in \CZ(\End_M(\sigma)), \forall \sigma \in \text{Obj}(\fkR_{\rho_M})$. 
One case of interest: 
If $\Sigma =\ind_{J_M}^M(\rho_M)$ then $z_{\Sigma}\in \CZ(\CH(M,\rho_M))$.

The equivalence of categories (Definition \ref{def:type}) $\mathfrak{M}_{\rho_M}$ induces a ring isomorphism
\[m_{\rho_M}\colon\mathfrak{Z}^{[\rho_M]} \iso \CZ(\CH(M,\rho_M)),\quad z=(z_\sigma)_{\sigma \in \fkR_{\rho_M}} \mapsto z_{\ind_{J_M}^M(\rho_M)}.\]

Let $(J, \rho)$ be a $G$-cover of $(J_M,\rho_M)$. 
We know by \cite[Theorem 8.3]{BK98} that $(J, \rho)$ is an $\fkS(J,\rho)$-type and Theorem \ref{genpro54} gives an explicit description for this set:  
\[\fkS(J,\rho)=\{[M,\sigma\otimes \bar\nu]_G \colon \nu \in X({{}^\natural J_M}/J_M)\} .\]
Therefore, we also have an isomorphism of rings 
\[m_{\rho}\colon\mathfrak{Z}^{[\rho]} \iso \CZ(\CH(G,\rho)),\quad z=(z_\sigma)_{\sigma \in \fkR_{\rho}} \mapsto z_{\ind_{J}^G(\rho)}.\]
\subsubsection{}
In this section, we define an action of $N_G(\rho_M)$ on $\fkZ^{[\rho_M]}$ and by transport of structure we get an action on $\CZ(\CH(M,\rho_M))$ that is compatible with the isomorphism $m_{\rho_M}$.


Let $n\in N_G(M)$, Proposition \ref{Gequiv} shows that $n$ normalizes $[\rho_M]$ if and only if $(\ind_{J_M}^M(\rho_M))^n\simeq \ind_{J_M}^M(\rho_M)$. 
Accordingly, for any $n \in N_G(\rho_M)$ and any $(\sigma,\CV)\in \text{Obj}(\fkR_{\rho_M})$ we clearly have 
$\sigma^n \in \text{Obj}(\fkR_{\rho_M})$. 
\begin{itemize}
\item
For any $n \in N_G(\rho_M)$ and $z\in \mathfrak{Z}^{[\rho_M]}$, define
the following map:
\[n\cdot z = ((n\cdot z)_\sigma:=z_{\sigma^{n^{-1}}})_{\sigma \in \fkR_{\rho_M}}.\] 
This defines an action of $N_G(\rho_M)$ on $\mathfrak{Z}^{[\rho_M]}$. 
Now if $m\in M$, then we have a commutative diagram 
\[	
	\begin{tikzcd}
	\sigma\arrow{r}{r_m } \arrow[swap]{d}{z_{\sigma}} &  \sigma^{m^{-1}}\arrow{d}{z_{\sigma^{m^{-1}}}} \\
	\sigma\arrow{r}{r_m}  &  \sigma^{m^{-1}}
	\end{tikzcd}
\]
where $r_m$ is the isomorphism given by $ v \mapsto \sigma(m^{-1})(v)$. 
It follows readily that $z_\sigma =z_{\sigma^{m^{-1}}}$. 
Accordingly, the defined action of $N_G(\rho_M)$ factors through $W_{[\rho_M]}$. 
\item Write $\CW$ for the underlying space of  $\Sigma=\ind_{J_M}^M(\rho_M)$. 
For any $n\in N_G(\rho_M)$ choose an element $\bar w \in \text{Aut}_\BC(\CW)$ which realizes the isomorphism $\Sigma \iso  \Sigma^{n^{-1}} $. 
Given an element $z \in \fkZ(\fkR_{\rho_M})$, the following diagram is by definition commutative. 
\[	
	\begin{tikzcd}
	\CW\arrow{r}{\bar w} \arrow[swap]{d}{z_{\Sigma}} &\CW\arrow{d}{z_{\Sigma^{n^{-1}}}} \\
	\CW\arrow{r}{\bar w}  &  \CW
	\end{tikzcd}
\]
\end{itemize}
Thus $m_{\rho_M}(n\cdot z)=(n\cdot z)_{\Sigma}={\bar w}\circ z_{\Sigma}\circ {\bar w}^{-1}={\bar w}\circ m_{\rho_M}(z)\circ {\bar w}^{-1}$. 
So by transport of structure we get the following action on the center of the Hecke algebra:
\[n\cdot \psi={\bar w}\circ \psi \circ {\bar w}^{-1},\quad \forall n\in W_{[\rho_M]}, \forall \psi\in\CZ(\CH_M(M,\rho_M)).\]
 }
Finally, we fix the isomorphism $ \mu_M: \BC[{}^\dag J_M/J_M] \rightarrow  \CZ(\CH(M, \rho_M)) $ as in Theorem \ref{centerhecke} (which is not canonical). We use this isomorphism to give an action of $W_{[\rho_M]}$ on $\BC[{}^\dag J_M/J_M]$ by transport of structure. 
\begin{theorem}\label{Satakegen} 
Assume ${}^\flat M= {}^0 M$. Let $(J_M,\rho_M)\in \CJ_M^{wc}$ and $(J,\rho)$ a $G$-cover. 
We have the following isomorphism of $\BC$-algebras
\[\CZ(\CH(G,\rho))=\CZ(\CH(M,\rho_M))^{W_{[\rho_M]}}\simeq \BC[{}^\dag J_M/J_M]^{W_{[\rho_M]}},\]
where, the first is canonical while the second is not.
\end{theorem}
\begin{proof}

By \cite[Theorem 1.9.1.1]{Roc09} we know that 
\[\fkZ^{\fks} = (\fkZ^{\fkt})^{W_{\fkt}},\]
where, $\fkt:=[M,\pi_M]_M= \fkS( {J_M}, \rho_M)$,  $\fks:=[M,\pi_M]_G$ and $W_{\fkt}=W_{[\rho_M]}$. 

By ${W_{[\rho_M]}}$-equivariance of $m_{[\rho_M]}$, we get a canonical isomorphism
\[\CZ(\CH(G,\rho))= (\CZ(\CH(M,\rho_M)))^{W_{[\rho_M]}}.\]
And finally, we conclude using the ${W_{[\rho_M]}}$-equivariance of $\mu_M$.
\end{proof}

\section{Some nice families of compact open subgroups}

Let $K$ be a compact open subgroup of $G$ and let $\fkR_K(G)$ be the full sub-category of $\fkR(G)$ consisting of representations $(\pi, V)$ that are generated by their $K$-fixed vectors. Write $\CH(G,K)$ for the Hecke algebra $\CH(G,1_K)$, where $1_K$ denotes the trivial representation of $K$.

Let $\bf S$ be a maximal split torus in $\bf G$. In \cite[Section 3.7]{BD84} the authors introduce criteria on $K$, which we call $\heart_S$ and recall now. 
\begin{definition}\label{heart}
Let $K$ be a compact open subgroup of $G$. We say $K$ satisfies $\heart_S$ if
\begin{enumerate}
\item Let $\bf P$ be a parabolic subgroup of ${\bf G}$ that contains $\bf S$. Write  $\bf P = MN$ with Levi component $\bf M$ and unipotent radical $\bf N$. Let $K'$ be a $G$-conjugate of $K$ and let $K'_P = K' \cap P/K'\cap N$. For any parabolic subgroup $\bf Q$ of $\bf G$   with the same Levi subgroup $\bf M$ and any other $G$-conjugate $K_1$ of $K$, $(K_1)_Q$ is a conjugate of $K'_P$ in $M$. 
\item Let $(\sigma, V)$ be a representation of $G$.  Let $V(N) = Span\langle \sigma(n)v - v | v \in V, n \in N \rangle$ and let $V_N = V/V(N)$. Let $V^K$ be the set of $K$-fixed vectors of $V$. Then the canonical map $V^{ K} \rightarrow V_N^{M \cap K}$ is surjective.

\end{enumerate}
\end{definition}
Let $\CK^\heart(S,G)$ be the collection of all compact open subgroups of $G$ that satisfies $\heart_S$. Let us recall the following proposition. 
\begin{proposition} [Corollary 3.9 of \cite{{BD84}}]\label{categories}
Let $\bf S$ be a maximal $F$-split torus in $\bf G$ and let $K \in \CK^\heart(S, G)$. 
The pair $(K,1)$ is $\fkS(K)$-type for a finite set $\fkS(K):= \fkS(K,1) \subset \fkB(G)$ (see Definition \ref{def:type}). 
 \end{proposition}

\begin{lemma}\label{BuProp}
Let $K\in \CK^\heart(S, G)$. We have $\fks = [ M, \sigma]_G \in \fkS(K)$ if and only if $\sigma^{K \cap M} \neq 0$. 
\end{lemma}
\begin{proof}
The proof given in \cite[Proposition 4]{Bu01}  goes through verbatim.
\end{proof}
\subsection{Some compact open subgroups that live in $\CK^\heart(S, G)$}
In \cite[Section 5]{BS20}, the following condition is considered in place of $\heart_S$ above. 

\begin{definition}\label{spade} 
 Let $\bf S$ be a maximal $F$-split torus in $\bf G$. Let $K$ be a compact open subgroup of $G$ and let $K^G$ be the set of $G$-conjugates of $K$. We say $K$ satisfies $\spade_S$ if, for any parabolic subgroup $\bf P$ of $\bf G$ that contains $\bf S$, any $P$-conjugacy class of $K^G$ contains a $K'$ that admits an Iwahori decomposition with respect to $P$:
    \[K' = (K'\cap N^-)(K'\cap M)(K'\cap N).\]
\end{definition}
Let $\CK^\spade(S, G)$ be the collection of compact open subgroups of $G$ that satisfy $\spade_S$. It is shown in \cite[Proposition 5.1]{BS20} that Proposition \ref{categories} holds for all $K \in \CK^\spade(S, G)$. 

Let $F_s$ be a separable closure of $F$ and let $\bF$ be the completion of the maximal unramified extension of $F$ contained in $F_s$. Let $\CB(G, F)$ (resp. $\CB({\bf G}_\bF, \bF)$) denote the Bruhat-Tits building of $\bf G$ over $F$ (resp. ${\bf G}_\bF$ over $\bF$). Let $\CA(S,F)$ denote the apartment of $\bf S$ over $F$. For $r \in \BR_{\geq 0}$ and $x \in \CB(G, F)$ let $G_{x,r}$ denote the Moy-Prasad filtration subgroup (see \cite{MP1, MP2}).  By \cite[Proposition 5.2]{BS20}, we have $G_{x,r}\in \CK^\spade(S,G) $ for all $x \in \CA(S,F)$ and $r>0$.

We are interested in compact open subgroups for which Lemma \ref{BuProp} holds, that is in compact open subgroups that lie in $\CK^\heart(S,G)$. Before taking this up, let us recall some preliminaries about filtrations of root subgroups from \cite[Chapter 4 and Chapter 5]{BT2}. 
\subsubsection{Filtration of root subgroups}\label{subsec:filtrations} Recall that we have fixed a valuation $\omega$ on $F$ so that $\omega(F^\times) = \BZ$.   Let $\bf G$ be a connected, reductive group over $F$. Then by \cite{Ste65}, ${\bf G}_{\bF}$  is quasi-split. Let $\sigma$ denote the Frobenius action on ${\bf G}_\bF$ so that ${\bf G} = {\bf G}_{\bF}^\sigma$. Let $\tF$ be the smallest sub-extension of the completion of $F_s$ over which ${\bf G}_{\bF}$ splits. Let $\bf S$ be a maximal $F$-split torus in $\bf G$ and let $\bS$ be a maximal $\bF$-split $F$-torus containing $\bf S $. Let ${\bf T }= Z_{\bf G}(\bS)$. Then ${\bf T}$ is a maximal torus in ${\bf G}$. Let $\bal$ be a $\sigma$-stable alcove in the apartment $\CA(\breve S, \bF)$ and let $\ca = \bal^\sigma$. Then $\ca$ is an alcove in the apartment $\CA(S, F)$. Choose a special vertex $x_0$ in the closure of $  \ca$. Let $\breve \Phi = \Phi( {\bf G}_{\bF}, \bS)$ denote the set of roots of $\bS$ in ${\bf G}_{\bF}$. Similarly we have $\Phi = \Phi(\bf G, \bf S)$. The choice of $x_0$ in the closure of $\ca$ determines a set of simple roots $\breve\Delta$ of $\breve\Phi$ and $\Delta$ of $\Phi$. Let $\breve\Phi^{\nd}$ denote the set of  non-divisible roots of $\breve\Phi$. We similarly have  $\Phi^\nd$. 

Let $W(G,S)$ be the Weyl group of $\bf G$ relative to $\bf S$. Let $W$ be the Iwahori Weyl group of $G$ over $F$. Having chosen the special vertex $x_0$, we may and do identify $W(G,S)$ with the subgroup of $W$ fixing $x_0$ (see \cite[Lemma 3.0.1(1)]{HR10}). In particular $W(G,S)$ acts on $\CA(S,F)$.

Let us first recall the definition of filtration of roots subgroups for $\ba \in \breve \Phi$. We will then prove a lemma that describes the set of jumps. Let ${\bf U}_\ba$ be the root subgroup attached to $\ba$. There are two possibilities.  
\begin{enumerate}
    \item Suppose $ \ba \in \breve \Phi^{\nd}$ is such that $2 \ba$ is not a root. We fix a pinning $(L_\ba, x_\ba)$ as in \cite[Section 4.1.5 and Section 4.1.8]{BT2}. Here $L_\ba \hookrightarrow \tF$ and $x_\ba: {\bf U}_\ba \xrightarrow{\simeq} \Res_{L_\ba/\bF}\BG_a$ is an isomorphism. Let $e_{\ba} = [L_{\ba}:\bF]$. Let $ \Gamma_{\ba}' = \omega(L_\ba^\times) = \frac{1}{e_\ba} \BZ$. The set of affine roots with gradient $\ba$ are of the form $\ba +m$ for $ m \in \Gamma_\ba'$. For $m = \frac{k}{e_\ba} \in \Gamma_{\ba}'$, we have  ${\bf U}_\ba(\bF)_{x_0, m} = x_\ba^{-1}(\fkp_{L_{\ba}}^{k})$ and for any real number $r$ let  ${\bf U}_\ba(\bF)_{x_0, r}={\bf U}_\ba(\bF)_{x_0, m}$ where $m = inf\{h \in \Gamma_\ba'\;|\; h \geq r\}$. 
    \item Suppose $ \ba \in \breve \Phi^{\nd}$ is such that $2 \ba$ is a root. We fix a pinning $(L_{\ba}, L_{2 \ba}, x_\ba)$ as in \cite[Section 4.1.9]{BT2}. Here $L_\ba$ is a quadratic extension of $L_{2\ba}$ with unique non-trivial automorphism $\tau$, $H_0(L_\ba, L_{2 \ba}) = \{(u, v) \in L_\ba \times L_\ba\;|\; v +\tau(v) = u \tau(u)\}$ with multiplication given by \cite[Equation (4) of Section 4.1.9]{BT2}. Let $H(L_\ba, L_{2\ba}) = \Res_{L_{2 \ba}/\bF} H_0(L_\ba, L_{2 \ba})$ and let $x_\ba: {\bf U}_\ba \xrightarrow{\simeq} H(L_\ba, L_{2\ba})$. Let $e_\ba = [L_\ba:F]$ and $e_{2\ba} = [L_{2\ba}:F]$. Note that $e_\ba = 2e_{2\ba}$. As in \cite[Lemma 4.3.3]{BT2}, let $L_\ba = L_{2\ba}[t]$ where $t^2 - \alpha t +\beta=0$.  If $\alpha=0$, set $\lambda = \frac{1}{2}$. If $\alpha \neq 0$, set $\lambda = t\alpha^{-1}$. Let $\Gamma_\ba'$ be the value set attached to the root $\ba$ as in \cite[Section 4.2.20]{BT2}. Note that $\omega(L_\ba^\times) = \frac{1}{e_\ba}\BZ $. Then by \cite[Section 4.3.4]{BT2},
    \begin{align}\label{GammabF}\Gamma_\ba' = \begin{cases} 
  \frac{1}{e_\ba}\BZ & \text{if  $\alpha=0$}\\
 \frac{1}{2e_\ba} +\frac{1}{e_\ba}\BZ & \text{if $\alpha\neq 0$.}
 \end{cases} \end{align}
Let $L_\ba^0$ be the set of elements in $L_\ba$ of trace 0. Then again by \cite[Section 4.3.4]{BT2}, 
$\Gamma_{2\ba}'  = \omega(L_{\ba}^0 \backslash \{0\})$. Let $\gamma =-\frac{1}{2}\omega(\lambda)$. For $m\in \Gamma_\ba'$, let (see \cite[Section 4.3.5]{BT2})
\[{\bf U}_\ba(\bF)_{x_0,m} = \left\{x_\ba(u,v) \in H(L_\ba, L_{2\ba})\;|\; \omega(u)\geq m+\gamma,\;  \omega(v - \lambda u \tau(u)) \geq 2m+ \frac{1}{e_a}\right\}.\]
This definition is extended to $r \in \BR$ as in \cite[Section 4.3.8]{BT2}. 
\end{enumerate}

\begin{lemma}\label{lem:filbF}
Let $\ba \in \breve\Phi^{\nd}$. Let $m \in \frac{1}{e_\ba} \BZ$. Let $r \in \BR$ be such that $0<r<\frac{1}{e_\ba}$ if $2\ba$ is not a root, and such that $0<r<\frac{1}{2e_\ba}$ if $2\ba$ is a root. Then ${\bf U}_{\ba}(\bF)_{x_0, m-r} = {\bf U}_{\ba}(\bF)_{x_0, m}$ and  ${\bf U}_{\ba}(\bF)_{x_0, m+r} = {\bf U}_{\ba}(\bF)_{x_0, m+\frac{1}{e_\ba}}$.
\end{lemma}
\begin{proof}
     
 Write $m = \frac{k}{e_\ba}$ for a suitable $k \in \BZ$.
 
    Let $\ba \in (\breve\Phi)^{\nd}$ be such that $2\ba$ is not a root. Then $\Gamma_\ba' = \frac{1}{e_\ba}\BZ$. We have assumed that $0<r< \frac{1}{e_\ba}$. So ${\bf U}_{\ba}(\bF)_{x_0, m - r} = \fkp_{L_\ba}^{k} = {\bf U}_{\ba}(\bF)_{x_0,m}$ and ${\bf U}_{\ba}(\bF)_{x_0, m+r} = \fkp_{L_\ba}^{k+1} = {\bf U}_{\ba}(\bF)_{x_0,m+\frac{1}{e_\ba}}$.
    
    Next, let  $\ba \in \breve\Phi$ be such that $2\ba$ is a root.  We have a few cases.
\begin{enumerate}
    \item Suppose $\alpha=0$. Then $\lambda = \frac{1}{2}$. We have two subcases. 
    \begin{itemize}
        \item Suppose the residue characteristic of $F$ is not 2. Then $\omega(\lambda) = 0$. We have assumed that $0<r<\frac{1}{2e_\ba}$. Now, ${\bf U}_{\ba}(\bF)_{x_0, m} =\{ x_\ba(u,v) \in H(L_\ba, L_{2\ba})\;|\; \omega(u) \geq \frac{k}{e_\ba}, \; \omega(v - \lambda u \tau(u)) \geq \frac{2k}{e_\ba}+\frac{1}{e_\ba}\}$. Using the facts that 

\[\frac{k}{e_\ba}-\frac{1}{2e_\ba}< \frac{k}{e_\ba}-r<\frac{k}{e_\ba}\] and \[\frac{2k}{e_\ba}< \frac{k}{e_\ba} +\frac{1}{e_\ba} -2r < \frac{2k}{e_\ba}+\frac{1}{e_\ba},\] it follows that ${\bf U}_{\ba}(\bF)_{x_0, m-r} = {\bf U}_{\ba}(\bF)_{x_0, m}$ and that ${\bf U}_{\ba}(\bF)_{x_0, m} = {\bf U}_{\ba}(\bF)_{x_0, m+\frac{1}{e_\ba}}$.
\item Suppose the residue characteristic of $F$ is 2. Note that the characteristic of $F$ is necessarily 0. So $\omega(\lambda) = -\omega(2) = -e_F$, where $e_F$ is the ramification index of $F/\BQ_2$. We have assumed that $0<r<\frac{1}{2e_\ba}$.  We have
\[{\bf U}_\ba(\bF)_{x_0,m} = \left\{x_\ba(u,v) \in H(L_\ba, L_{2\ba})\;|\; \omega(u)\geq \frac{k}{e_\ba}+\frac{e_F}{2},\;  \omega(v - \lambda u \sigma(u)) \geq \frac{2k}{e_\ba}+ \frac{1}{e_\ba}\right\}.\]
. Write $\frac{e_F}{2} = \frac{e_F e_{2\ba}}{e_\ba}$. Then we see that \[\frac{k}{e_\ba}+\frac{e_F e_{2\ba}}{e_\ba} - \frac{1}{e_\ba}< \frac{k}{e_\ba}+\frac{e_F e_{2\ba}}{e_\ba} - r < \frac{k}{e_\ba}+\frac{e_F e_{2\ba}}{e_\ba}\]  and \[\frac{2k}{e_\ba}<\frac{2k}{e_\ba}+\frac{1}{e_\ba} - 2r< \frac{2k}{e_\ba}+\frac{1}{e_\ba}.\] It now again follows that ${\bf U}_{\ba}(\bF)_{x_0, m-r} = {\bf U}_{\ba}(\bF)_{x_0, m}$ and that ${\bf U}_{\ba}(\bF)_{x_0, m+r} = {\bf U}_{\ba}(\bF)_{x_0, m+\frac{1}{e_\ba}}$.
    \end{itemize}
\item Suppose $\alpha \neq 0$. We have assumed $0<r<\frac{1}{2e_\ba}$. Then we have
\[{\bf U}_\ba(\bF)_{x_0,m} = \left\{x_\ba(u,v) \in H(L_\ba, L_{2\ba})\;|\; \omega(u)\geq \frac{k}{e_\ba}-\frac{1}{2}\omega(\lambda),\;  \omega(v - \lambda u \sigma(u)) \geq \frac{2k}{e_\ba}+ \frac{1}{e_a}\right\}.\]
Now, write $\frac{k}{e_\ba}-\frac{1}{2}\omega(\lambda) = \frac{1}{2e_\ba} + \frac{k'}{e_\ba}$. Then, for $0<r<\frac{1}{2e_\ba}$, we have \[ \frac{k'}{e_\ba} <  \frac{1}{2e_\ba} + \frac{k'}{e_\ba}-r <  \frac{1}{2e_\ba} + \frac{k'}{e_\ba} < \frac{k'}{e_\ba} +\frac{1}{e_\ba} \]
and
\[ \frac{2k}{e_\ba} < \frac{2k}{e_\ba} + \frac{1}{e_\ba} -r< \frac{2k}{e_\ba} + \frac{1}{e_\ba}.\]
Now, it is again clear that ${\bf U}_\ba(\bF)_{x_0,m-r} = {\bf U}_\ba(\bF)_{x_0,m}$ and that ${\bf U}_\ba(\bF)_{x_0,m+r} = {\bf U}_\ba(\bF)_{x_0,m+\frac{1}{e_\ba}}$.
    
\end{enumerate}

This finishes the proof of the lemma.
\end{proof}

Next, we recall the definition of the filtration of the root subgroup ${\bf U}_a(F)$ for $a \in \Phi$ (cf. \cite[\S 5.1.16 - 5.1.18]{BT2}). Let $\breve\Phi_a: = \{ \bc \in \breve\Phi\;|\; \bc|_{\bf S} = a \text{ or } 2a\}$. This is a $\s$-stable positively closed subset of $\breve\Phi$; that is if $\bc_1, \bc_2 \in \breve\Phi_a$ are such that $\bc_1+\bc_2$ is a root, then $\bc_1+\bc_2 \in \breve\Phi_a$. For any fixed ordering, the subset
\begin{align}\label{ubr}
    {\bf U}_a(\bF)_{x_0,r} := \prod_{\bc \in \breve\Phi_a, \bc|_{\bf S} = a} {\bf U}_\bc(\bF)_{x_0,r}  \prod_{\bc \in \breve\Phi_a^{\nd}, \bc|_{\bf S} = 2a} {\bf U}_\bc(\bF)_{x_0,2r}
\end{align}
is a subgroup of ${\bf U}_a(\bF)$. Let ${\bf U}_a(F)_{x_0,r} := {\bf U}_a(\bF)_{x_0,r} \cap {\bf U}_a(F)$. Let $\Gamma_a'$ be the value set attached to the root $a$ as in \cite[Section 5.1.16]{BT2}. Let $\ba \in \breve\Phi_a$ be such that $\ba|_{\bf S} = a$. Then by \cite[Proposition 5.1.19]{BT2}, we have $\Gamma_a' = \Gamma_\ba'$. For $a \in \Phi$ and $\ba \in \breve\Phi$ such that $\ba|_{\bf S} = a$, define $e_a := e_\ba$. Note that this definition does not depend on the choice of $\ba$ whose restriction to $\bf S$ is $a$ (see \cite[Section 5.1.15]{BT2}). We have the following lemma.

\begin{lemma}\label{JumpsF} Let $a \in \Phi^{\nd}$ and let $\ba \in \breve\Phi$ such that $\ba|_S = a$. Let $m \in \frac{1}{e_a} \BZ$. Let $r \in \BR$ be such that $0<r<\frac{1}{e_a}$ if $2a$ is not a root, and such that $0<r<\frac{1}{2e_a}$ if $2a$ is a root. Then ${\bf U}_{a}(F)_{x_0, m-r} = {\bf U}_{a}(F)_{x_0, m}$ and  ${\bf U}_{a}(F)_{x_0, m+r} = {\bf U}_{a}(F)_{x_0, m+\frac{1}{e_a}}$.
\end{lemma}
\begin{proof}
   Let $m \in \frac{1}{e_a}\BZ$ and $r$ as above. We need to show that ${\bf U}_a(\bF)_{x_0,m - r} ={\bf U}_a(\bF)_{x_0,m}$. Let $\bc \in \breve\Phi_a$ be such that $\bc|_{\bf S} =a$. Then, since $\Gamma_{\bc}' = \Gamma_a'$ by \cite[Proposition 5.1.19]{BT2} and $e_\bc = e_a$, Lemma \ref{lem:filbF} implies that ${\bf U}_\bc(\bF)_{x_0,m - r} ={\bf U}_\bc(\bF)_{x_0,m}$ and that ${\bf U}_\bc(\bF)_{x_0,m + r} ={\bf U}_\bc(\bF)_{x_0,m+\frac{1}{e_a}}$ . Next, suppose  $\bc \in \breve\Phi_a^{\nd}$ is such that $\bc|_{\bf S} =2a$. Then there exist distinct $\bc_1, \bc_2 \in \breve\Phi$ such that $\bc_1+\bc_2 = \bc$ and $\bc_1|_{\bf S} = \bc_2|_{\bf S} = a$. Further $\Gamma_{\bc_1}' = \Gamma_{\bc_2}' =  \Gamma_{\bc}' = \Gamma_{2a}' =\Gamma_a'$ and $e_{\bc_1} = e_{\bc_2} = e_{\bc} = e_{{2a}} = e_a$. Since we have assumed that $0<2r<\frac{1}{e_a}$, we see that 
   \[2m-\frac{1}{e_a} < 2m-2r< 2m \;\; \text{ and }\;\; 2m < 2m+2r< 2m+\frac{1}{e_a}\]
   So ${\bf U}_{\bc}(\bF)_{x_0, 2m-2r} = {\bf U}_{\bc}(\bF)_{x_0, 2m}$ and similarly, ${\bf U}_{\bc}(\bF)_{x_0, 2m+2r} = {\bf U}_{\bc}(\bF)_{x_0, 2m+\frac{1}{e_a}}$. This proves that ${\bf U}_a(\bF)_{x_0,m - r} ={\bf U}_a(\bF)_{x_0,m}$ and that ${\bf U}_{a}(F)_{x_0, m+r} = {\bf U}_{a}(F)_{x_0, m+\frac{1}{e_a}}$. This finishes the proof of the lemma.
\end{proof}
Let \begin{align}
    \BR_G := \{r \in \BR_{\geq 0}\;|\; r \in \frac{1}{e_a}\BZ \text{ for all } a \in \Phi\}.
\end{align}
\begin{remark}
    Note that $\BN \subset \BR_G$. For example, if $\bf G$ is a connected, reductive group that splits over an unramified extension of $F$, we have $\BR_G=\BN$. If ${\bf G} = \Res_{L/F} {\bf G'}$, then $\BR_G = \frac{1}{e_{L/F}} \BR_{G'}$ where $e_{L/F}$ is the ramification index of $L/F.$
\end{remark}

\begin{proposition}\label{PropIm}
    Let $\ca$ be an alcove in $\CA(S,F)$. Let $x \in \ca$ and let $m\in \BR_G$. Then $G_{x, m} \in \CK^\heart(S,G)$.
\end{proposition}

\begin{proof}
 We only need to verify that $G_{x,m}$ satisfies (1) of Definition \ref{heart}.  Let $N$ be the normalizer of $S$ in $G$. Then, using the Iwasawa decomposition, we know that $g G_{x,m}g^{-1}$ is $P$-conjugate to $n G_{x,m}n^{-1}$ for a suitable $n$ in $N$. But $n G_{x,m} n^{-1} = G_{n(x), m}$. So, we only need to verify that $G_{n(x),m} \cap M$ is $M$-conjugate to $G_{x,m} \cap M$.   We may and do assume that $M = M_\theta$ for a suitable $\theta \subset \Delta$. Let $\Phi_\theta$ be the set of roots in $\Phi(G,S)$ that lie in the $\BQ$-span of $\theta$. We accordingly have $\Phi^+_\theta$ and $\Phi^-_\theta$.  Then $W_\theta = \langle s_a\;|\; a \in \theta\rangle = W(M,S)$. Every element $w \in W(G,S)$ can be written as $w_1w_2$ where $w_1 \in W_\theta$ and $w_2^{-1}(\theta)>0$. 

To prove (2), it suffices to show that $G_{w_1w_2\cdot x, m} \cap M$ is $M$-conjugate to $G_{x,m} \cap M$. Since $w_1 \in W_\theta$, we see that $G_{w_1w_2\cdot x, m}\cap M$ is $M$-conjugate to $G_{w_2\cdot x,m}\cap M$. Hence we only need  to show that $G_{w_2\cdot x, m}\cap M = G_{x,m}\cap M$. Since $G_{x,m} = \langle T_m, {\bf U}_{a}(F)_{x,m}\;|\; a \in \Phi_\theta\rangle$, it suffices show that \begin{align}\label{propclaim}
    {\bf U}_{a}(F)_{ w_2\cdot x, m} = {\bf U}_{a}(F)_{x,m}\;\; \forall a \in \Phi_\theta. 
\end{align} But ${\bf U}_{a}(F)_{x, m}={\bf U}_{a}(F)_{ x_0, m - a(x-x_0)}$ and $ {\bf U}_{a}(F)_{ w_2\cdot x, m} = {\bf U}_{a}(F)_{x_0, m - w_2^{-1}(a)(x-x_0)}$. 

Let $s_a = \inf\{s \in \Gamma_a'\;|\; s>0\}$. Consider the affine linear functional $\psi_{a, r}:y \rightarrow a(y-x_0) +r$ for $a \in \Phi, r \in \BR$. This is an affine root precisely when $r \in \Gamma_a'$ by \cite[Proposition 4.2.22 and Theorem 5.1.20]{BT2}. Having chosen $\ca$ and $x_0$, we see that for $a \in \Phi^+$, $\psi_{a, 0}$ and $\psi_{-a, s_a}$ are both positive affine roots. Since $x \in \ca$, we see that $0<a(x - x_0)<s_a$. 

Note that $\Gamma_{w_2^{-1} a}' = \Gamma_a'$. For $a \in \Phi_\theta^+$, since $w_2^{-1}(a)$ is 
 positive, we have $\psi_{w_2^{-1}(a), 0}$ and $\psi_{-w_2^{-1}(a), s_a}$ are also positive affine roots, so $0<w_2^{-1}(a)(x - x_0)<s_a$.

To prove \eqref{propclaim}, we need to show  that for $a \in \Phi_\theta^+$, \begin{align}\label{Prop:Eq1}{\bf U}_{a}(F)_{ x_0, m - a(x-x_0)} = {\bf U}_{a}(F)_{x_0, m - w_2^{-1}(a)(x-x_0)} = {\bf U}_{a}(F)_{x_0, m},\end{align} and that for $a \in \Phi_\theta^-$, \begin{align}\label{Prop:Eq2} {\bf U}_{a}(F)_{ x_0, m - a(x-x_0)} = {\bf U}_{a}(F)_{x_0, m - w_2^{-1}(a)(x-x_0)} = {\bf U}_{a}(F)_{x_0, m+\frac{1}{e_a}}.\end{align} We see that \eqref{Prop:Eq1} and \eqref{Prop:Eq2}  would follow from Lemma \ref{JumpsF} as soon as we show that for $a \in \Phi_\theta^+ \cap \Phi_\theta^{\nd}$, $a(x-x_0), w_2^{-1}(a)(x-x_0) \in (0, \frac{1}{e_a})$ if $2a$ is not a root and that $ a(x-x_0), w_2^{-1}(a)(x-x_0) \in (0,\frac{1}{2e_a})$ if $2a$ is a root. We have a few cases.
\begin{itemize}
    \item Suppose $a \in \Phi_\theta^{\nd} \cap \Phi_\theta^+$ is such that $2a$ is not a root. Then $\Gamma_a' = \frac{1}{e_a}\BZ$, so $s_a = \frac{1}{e_a}$. So $a(x - x_0), w_2^{-1}(a)(x-x_0) \in (0, \frac{1}{e_a})$.
    \item Suppose $a \in \Phi_\theta^{\nd} \cap \Phi_\theta^+$ is such that $2a$ is a root and such that $\breve\Phi_a^{\nd} = \breve\Phi_a$. This means that there exist distinct $\ba_1, \ba_2 \in \breve\Phi_a$ such that $\ba_1|_{\bf S} = \ba_2|_{\bf S} = a$. Then $\Gamma_a' = \Gamma_{2a}' =\frac{1}{e_a}\BZ$, so $s_a = s_{2a} = \frac{1}{e_a}$. So $a(x - x_0), w_2^{-1}(a)(x-x_0)\in (0, \frac{1}{e_a})$ and $2a(x - x_0),  w_2^{-1}(2a)(x-x_0)\in (0, \frac{1}{e_a})$. In particular, $a(x - x_0),  w_2^{-1}(a)(x-x_0) \in (0, \frac{1}{2e_a})$. 
    \item Suppose $a \in \Phi_\theta^{\nd} \cap \Phi_\theta^+$ is such that $2a$ is a root and such that $\breve\Phi_a^{\nd} \subsetneq \breve\Phi_a$. Let $\ba \in \breve\Phi_a$ be such that $\ba|_{\bf S} = a$ and $2\ba$ is a root.  As recalled in subsection \ref{subsec:filtrations}, let $L_\ba = L_{2\ba}[t]$ where $t^2 - \alpha t+\beta=0$.
    \begin{itemize}
    \item Suppose $\alpha =0$. By \eqref{GammabF}, $\Gamma_\ba' = \frac{1}{e_\ba}\BZ$. Further, $\Gamma_{2\ba}' = \omega(L_\ba^0 \backslash \{0\}) = \frac{1}{e_\ba} +\frac{2}{e_\ba}\BZ$, so $s_\ba = s_{2\ba} = \frac{1}{e_\ba}$. Noting that $e_a = e_\ba$, we see that $a(x - x_0), w_2^{-1}(a)(x-x_0)\in (0, \frac{1}{e_a})$ and $2a(x - x_0),  w_2^{-1}(2a)(x-x_0)\in (0, \frac{1}{e_a})$. In particular, $a(x - x_0),  w_2^{-1}(a)(x-x_0) \in (0, \frac{1}{2e_a})$. 
\item Suppose $\alpha \neq 0$. Then $\Gamma_\ba' = \frac{1}{2e_\ba} +\frac{1}{e_\ba}\BZ$. So $s_\ba = \frac{1}{2e_\ba}$. Further, $\Gamma_{2\ba}' = \omega(L_\ba^0 \backslash \{0\}) = \frac{2}{e_\ba}\BZ$. Noting that $e_a = e_\ba$, we see that $a(x - x_0), w_2^{-1}(a)(x-x_0)\in (0, \frac{1}{2e_a})$.
\end{itemize}
\end{itemize}
We have proved \eqref{propclaim}. This finishes the proof of the proposition.
\end{proof}

\subsection{Some compact open subgroups that don't live in $\CK^\heart(S,G)$}\label{CounterEg} It is shown in \cite[Section 5]{BS20}, that for each $x\in\CA(S,F)$ and each $r>0$, $G_{x,r} \in \CK^\spade(S,G)$. In this subsection, we give examples of $G_{x,r}$'s that do not lie in $\CK^\heart(S,G)$. In the proof of Proposition \ref{PropIm}, we had used crucially that if $x$ is not already a special point, then it lies in an alcove \textit{and} that $r \in \BR_G$ for the argument to go through. This suggests how to look for points $x \in \CA(S,F)$ and $r>0$ for which $ G_{x,r} \notin \CK^\heart(S,G)$.

\subsubsection{Example} Let $G = \GL_3$ with the diagonal matrices as $T$ and upper-traingular matrices as $B$. With this choice, let $\Delta = \{e_1-e_2, e_2-e_3\}$. Let $M = \GL_1 \times \GL_2$. Then $\theta = \{e_2-e_3\}$. Let $a = e_2-e_3$, let $w = s_{e_1-e_2}$ and let $x=e_1^*/2$. Then
\[
 G_{x,1} = 
  \left[ {\begin{array}{ccc}
   1+\fkp_F & \fkp_F &\fkp_F \\
   \fkp_F^2 & 1+\fkp_F & \fkp_F \\
   \fkp_F^2 & \fkp_F &1+\fkp_F\\
  \end{array} } \right].
\]
With 
 \[
 n = 
  \left[ {\begin{array}{ccc}
   0 & 1 &0 \\
   1 & 0 & 0 \\
   0& 0 &1\\
  \end{array} } \right]
\]
we know that $n$ is a representative of $w$ in $\GL_3$. Now, 
\[
n G_{x,1}n^{-1} = 
  \left[ {\begin{array}{ccc}
   1+\fkp_F & \fkp_F^2 &\fkp_F \\
   \fkp_F & 1+\fkp_F & \fkp_F \\
   \fkp_F & \fkp_F^2 &1+\fkp_F\\
  \end{array} } \right].
\]
Then
\[
 G_{x,1} \cap M = 
  \left[ {\begin{array}{ccc}
   1+\fkp_F & 0 & 0 \\
   0 & 1+\fkp_F & \fkp_F \\
   0 & \fkp_F &1+\fkp_F\\
  \end{array} } \right]
\]
and 
\[
n G_{x,1}n^{-1} \cap M = 
  \left[ {\begin{array}{ccc}
   1+\fkp_F & 0 & 0 \\
   0 & 1+\fkp_F & \fkp_F \\
   0 & \fkp_F^2 &1+\fkp_F\\
  \end{array} } \right].
\]
We claim that $nG_{x,1}n^{-1} \cap M$ and $G_{x,1} \cap M$ are \textbf{not} $M$-conjugate. To see this, it suffices to prove that the groups 
\[
  K_1= \left[ {\begin{array}{cc}
    1+\fkp_F & \fkp_F \\
  \fkp_F &1+\fkp_F\\
  \end{array} } \right] \text{ and } I_1= \left[ {\begin{array}{cc}
   1+\fkp_F & \fkp_F \\
   \fkp_F^2 &1+\fkp_F\\
  \end{array} } \right]
\]
are not $\GL_2$-conjugate. This is intuitively clear since $K_1 \subset GL_2(\fkO_F)$ which corresponds to the parahoric subgroup of a hyperspecial vertex in the building and $I_1 \subset I = \left[ {\begin{array}{cc}
   \fkO_F^\times & \fkO_F \\
   \fkp_F &\fkO_F^\times\\
  \end{array} } \right]$ which is an Iwahori subgroup of $\GL_2(F)$ and corresponds to the parahoric subgroup of an alcove.   We justify this as follows. Normalize the Haar measure on $\GL_2(F)$ so that $\vol(I)=1$. Note that $K_1$ and $I_1$ are both normal subgroups of $I$ and $I_1 \subsetneq K_1 \subsetneq I$. Now, if $K_1$ and $I_1$ are $\GL_2$-conjugate, then $\vol(K_1) = \vol(I_1)$, which then implies that $[I:K_1] = \vol(K_1)^{-1} = \vol(I_1)^{-1} =[I:I_1]$, which is not possible. This proves that with $x=e_1^*/2$, $G_{x,1} \notin \CK^\heart(S, G)$.  \\
\subsubsection{Example}  Let $G = \GL_4$ with the diagonal matrices as $T$ and upper-traingular matrices as $B$. With this choice, let $a_i = e_i - e_{i+1}, \; 1 \leq i \leq 3$. Then $\Delta = \{a_1, a_2, a_3\}$. Let $M = \GL_2 \times \GL_2$. Then $\theta = \{a_1, a_3\}$. Let $w = s_{a_2}$ and let $\displaystyle{x=\frac{3a_1^\vee +4a_2^\vee +3a_3^\vee}{8}}$. Then $x$ is the barycenter of the alcove $\ca$ whose bounding hyperplanes are given by the affine roots $a_1, a_2, a_3, 1- (a_1 +a_2 +a_3)$. Let $r=3/8$. Then
\[
 G_{x,3/8} = 
  \left[ {\begin{array}{cccc}
   1+\fkp_F & \fkp_F &\fkO_F& \fkO_F\\
   \fkp_F & 1+\fkp_F & \fkp_F & \fkO_F\\
   \fkp_F & \fkp_F &1+\fkp_F& \fkp_F \\
   \fkp_F^2&\fkp_F & \fkp_F &1+\fkp_F \\
  \end{array} } \right].
\]
With 
 \[
 n = 
  \left[ {\begin{array}{cccc}
   1&0 & 0 &0 \\
   0& 0 & 1 & 0 \\
   0&1&0 &0\\
   0&0&0&1
  \end{array} } \right]
\]
we know that $n$ is a representative of $w$ in $\GL_4$. Now, 
\[
n G_{x,3/8}n^{-1} = 
  \left[ {\begin{array}{cccc}
   1+\fkp_F & \fkO_F &\fkp_F& \fkO_F\\
   \fkp_F & 1+\fkp_F & \fkp_F & \fkp_F\\
   \fkp_F & \fkp_F &1+\fkp_F& \fkO_F \\
   \fkp_F^2&\fkp_F & \fkp_F &1+\fkp_F \\
  \end{array} } \right].
\]
Then
\[
 G_{x,3/8} \cap M = 
  \left[ {\begin{array}{cccc}
   1+\fkp_F & \fkp_F &0& 0\\
   \fkp_F & 1+\fkp_F & 0 & 0\\
   0 & 0 &1+\fkp_F& \fkp_F \\
   0 &0 & \fkp_F &1+\fkp_F \\
  \end{array} } \right].
\]
and 
\[
n G_{x,3/8}n^{-1} \cap M = 
  \left[ {\begin{array}{cccc}
   1+\fkp_F & \fkO_F &0& 0\\
   \fkp_F & 1+\fkp_F & 0 & 0\\
   0 & 0 &1+\fkp_F& \fkO_F \\
   0 &0 & \fkp_F &1+\fkp_F \\
  \end{array} } \right].
\]
 Clearly, $G_{x, 3/8} \cap M$ and $n G_{x, 3/8} n^{-1} \cap M$ are not $M$-conjugate.
\section{The Bernstein center at deeper level}\label{BC}
Let $F$ be a non-archimedean local field and let $\bf G$ be a connected, reductive group over $F$. We assume the following.
\begin{assumption}\label{Fin}
$\bf G$ splits over a tamely ramified extension of $F$, and the residue characteristic $p$ of $F$ does not divide the order of the Weyl group of $\bf G$.
\end{assumption} 
\begin{enumerate}
    \item By \cite[Theorem 8.1]{Fin21} Every irreducible, supercuspidal representation of $G$ arises from Yu's construction, that was recalled in Section \ref{YuS}.
    \item Let $(\pi,V)$ be an irreducible, smooth representation of $G$ and let $\fks=[M,\sigma]_G$ be the inertial class of $\pi$. Let $(J_M, \rho_M)$ be a supercuspidal type of the Bernstein block corresponding to $\fks_M=[M,\sigma]_M$. Then there exists a $G$-cover $(J, \rho)$ of  $(J_M, \rho_M)$, which in particular says that $(J,\rho)$ is an $\fks$-type. This construction is carried out in \cite[Theorem 9.1]{KY17} under some additional hypothesis, but in \cite[Theorem 7.12]{Fin21}, the author has proved that this holds merely under Assumption \ref{Fin}.
\end{enumerate}

For the remainder of this section we assume that Assumption \ref{Fin} holds. We now recall Yu's construction of supercuspidal representations.

\subsection{Yu's construction of supercuspidal representations and a corollary}\label{YuS}  The Yu datum consists of a 5-tuple $(\vec{\bG}, y, \vec{r}, \rho_{-1}, \vec{\phi})$ where
\begin{enumerate}[\textbf{D}1:]
    \item $\vec{\bG} = (\bG^0, \cdots, \bG^d)$ is a tower of algebraic subgroups of $\bG$, $$\bG^0 \subsetneq \cdots \subsetneq \bG^d = \bG$$ such that $Z(\bG^0)/Z(\bG)$ is anisotropic over $F$ and  $\vec{\bG}$ is a tamely ramified twisted Levi sequence in $\bG$ in the section of \cite[Section 1]{Yu01}. In particular, $\bG^i \otimes F^t$ is split and is a Levi factor of a parabolic subgroup of $\bG\otimes F^t$.
    \item $y$ is a point in $\CB(\bG, F) \cap \CA(\bG, T, E)$ where $\bT$ is a maximal torus in $\bG^0$, $E$ is a Galois tamely ramified splitting extension of $T$ and hence of $\vec{\bG}$.
    \item $ \vec{r} = (r_0, \cdots, r_d)$ is a sequence of real numbers satisfying $0<r_0 \cdots <r_{d-1} \leq r_d$ if $d>0$ and $0 \leq r_0$ if $d=0$.
    \item $\rho_{-1}$ is an irreducible representation of $K^0 = G^0_{[y]}$ such that $\rho_{-1}|_{G^0_{y, 0+}}$ is 1-isotypic and the compactly induced representation $\pi_{-1} = \ind_{K^0}^{G^0} \rho_{-1}$ is irreducible, supercuspidal. Here $[y]$ is the projection of $y$ on the reduced building and $G^0_{[y]}$ is the subgroup of $G$ fixing $[y]$. 
    \item $\vec{\phi} = (\phi_0, \cdots, \phi_d)$ is such that each $\phi_i$ is a quasi-character of $G^i$ for each $i$. We assume that $\phi_i$ is trivial on $G^i_{r_i, +}$ but not on $G^i_{r_i}$ that is, that $depth(\phi_i) = r_i$ for $0 \leq i \leq d-1$. Here that we have used the convention $G^i_{r_i} = G^i_{y, r_i}$ and similarly for $G^i_{r_i+}$. If $r_{d-1}<r_d$ we assume that $\phi_d$ is trivial on $G^d_{r_d, +}$ but not on $G^d_{r_d}$, otherwise assume that $\phi_d=1$. We also assume that $\phi_i$ is $G^{i+1}$ generic (see \cite[Section 9]{Yu01}. 

\end{enumerate}
Starting with such a datum, Yu's construction gives a supercuspidal representation of depth $r_d$. Let us summarize this construction. Let $K^0_+ = G^0_{0+}$ and for $1 \leq i \leq d$, let $s_i = r_i/2$ and let 
\[K^i = K^0 G^1_{s_0} \cdots G^i_{s_{i-1}}\] and
\[ K^i_+ = K^0_+ G^1_{s_0+} \cdots G^i_{s_{i-1}+}\]
Yu also defines subgroups $J^i, J^i_+, 1 \leq i \leq d$ as follows. 
Let
\[J^i = \bG(F) \cap \langle {\bf U}_\alpha(E)_{y, r_{i-1}}, {\bf U}_\beta(E)_{y, r_{i-1}/2} \;|\; \alpha \in \Phi(\bG^{i-1}, T, E) \cup \{0\}, \beta \in \Phi(\bG^{i}, T, E) \backslash \Phi(\bG^{i-1}, T, E)\rangle\]
and 
\[J^i_+ =\bG(F) \cap \langle {\bf U}_\alpha(E)_{y, r_{i-1}}{\bf U}_\beta(E)_{y, r_{i-1}/2+} \;|\; \alpha \in \Phi(\bG^{i-1}, T, E) \cup \{0\}, \beta \in \Phi(\bG^{i}, T, E) \backslash \Phi(\bG^{i-1}, T, E)\rangle.\]
For $1 \leq i \leq d$, we have
\[K^{i-1} J^i = K^i,\;\; K^{i-1}_+ J^i_+ = K^i_+.\]
Yu's construction of the supercuspidal representation of $G$ from this data is done inductively and includes the following steps.
\begin{enumerate}[(a)]
    \item In \cite[Section 11]{Yu01}, for $0 \leq i \leq d-1$, Yu constructs an irreducible representation $\tilde\phi_{i-1}$ of $K^{i-1} \ltimes J^{i}$ using the character $\phi_{i-1}$ of $G^{i-1}$, that satisfies condition $ \textbf{SC2}_i$ in \cite[Section 4, Page 592]{Yu01}. Let us recall this construction. Let $\hat\phi_{i-1}$ be the character of $K^0G^{i-1}_0\bG(F)_{y, s_{i-1}+}$ as in \cite[Section 4, Page 591]{Yu01}. Using $\hat \phi_{i-1}$, he defines a non-degenrate $\BF_p$-valued pairing on the $\BF_p$-vector space $J^i/J^i_+$ making $J^i/J^i_+$ a symplectic space over $\BF_p$. Let $(J^i/J^i_+)^\#$ be the Heisenberg group of $J^i/J^i_+$. Yu constructs a canonical isomorphism 
\[j: J^i/(J^i_+ \cap \ker(\hat \phi_{i-1})\rightarrow (J^i/J^i_+)^\#\] in \cite[Proposition 11.4]{Yu01}.  Note that $K^i$  act on $J^i/J^i_+$ by conjugation and this gives a homomorphism from $K^i \rightarrow \text{Sp}(J^i/J^i_+)$. Let $\tilde\phi_{i-1}$ be the pull back of the Weil representation of $\text{Sp}(J^i/J^i_+) \ltimes (J^i/J^i_+)^\#$ via the map $K^{i-1} \ltimes J^i \rightarrow \text{Sp}(J^i/J^i_+) \ltimes (J^i/J^i_+)^\#$. He shows in \cite[Theorem 11.5]{Yu01} that $\tilde\phi_{i-1}|_{J^i_+}$  is $\hat\phi_{i-1}|_{J^{i}_+}$-isotypic and that the restriction of $\tilde\phi_{i-1}$ to $K^{i+1}_+$ is $1$-isotypic.
    \item Next, Yu constructs a representation $\rho_i'$ of $K^i$ such that $\rho_i'|_{G^i_{r_i}}$ is $1$-isotypic. He then sets $\rho_i= \rho_i' \otimes \phi_i|_{K^i}$. First, put $\rho_0' = \rho_{-1}$ and $\rho_0 = \rho_0' \otimes (\phi_0|_{K^0})$.  Now suppose that $\rho_{i-1}'$ and $\rho_{i-1}$ have already been constructed. Inflate $\phi_{i-1}|_{K^{i-1}}$ to a representation $\mathrm{inf}(\phi_{i-1})$ of $K^{i-1} \ltimes J^i$. He shows that the representation $\mathrm{inf}(\phi_{i-1}) \otimes \tilde\phi_{i-1}$ factors through the natural map $K^{i-1} \ltimes J^i \rightarrow K^{i-1}J^i =K^i$. Let $\phi_{i-1}'$ be the representation of $K^{i}$ whose inflation to $K^{i-1} \ltimes J^i$ is $\mathrm{inf}(\phi_{i-1}) \otimes \tilde\phi_{i-1}$.  Inflate $\rho_{i-1}'$ to a representation $\mathrm{inf}(\rho_{i-1}')$ of $K^i = K^{i-1} J^i$ via the map $K^i\rightarrow K^{i-1}J^i/J^i = K^{i-1}/K^{i-1} \cap J^i$ (This can be done because $\rho_{i-1}'$ restricted to $K^{i-1} \cap J^i$ is $1$-isotypic). Set $\rho_i' = \mathrm{inf}(\rho_{i-1}') \otimes \phi_{i-1}'$ and $\rho_i = \rho_i' \otimes (\phi_i|_{K^i})$. 
    \item The main theorem of Yu's paper \cite{Yu01} says that the compactly induced representation $\pi_i = \ind_{K^i}^{G^i} \rho_i$ of $G^i$ is irreducible and supercuspidal of depth $r_i$, $0 \leq i \leq d$. We note that the proof of this main theorem in Yu's paper relied on some propositions in literature that were later noted to be false. Recently, Fintzen gave an alternate proof of the main theorem in \cite{Fin21F}.    \end{enumerate}

Let $^0K^0 = G^0_y$ and $^0K^i = (^0K^0) G^1_{s_0} \cdots G^i_{s_{i-1}}$. Let $^0\rho_i$ be an irreducible summand of $\rho_i|_{^0K^i}$. As noted in \cite[Corollary 15.3]{Yu01}, we have that $(^0K^i, ^0\rho_i)$ is a $[G^i, \pi_i]_{G^i}$-type for $0 \leq i \leq d$.

\begin{lemma}\label{LemmaonYu} For each $0\leq i \leq d$, 
\begin{enumerate}
\item Let  $^0\rho_i$ be any irreducible summand of $\rho_i|_{^0K^i}$. Then every $g \in G^i$ that intertwines $^0\rho_i$ lies in $K^i$, and
\item $m_{^0K^0}(\rho_{-1}) = m_{^0K^i}(\rho_i).$
    
    \end{enumerate}
\end{lemma}
\begin{proof}  (1) is a consequence of \cite[Corollary 15.5]{Yu01}, whose corrected proof can be found in \cite[Proposition 4.4]{Oha21}.  In more detail, it is shown in loc. cit that if $g \in G^i$ intertwines $^0\rho_i$ then $g \in (^0K^i) G^0 (^0K^i)$ and that $g \in G^0$ intertwines $^0\rho_i$, then $g \in K^0$.  Hence $g \in G^i$ intertwines $^0\rho_i$, then $g \in (^0K^i)(K^0) (^0K^i)$. By definition $(K^0) (^0K^i) = K^i$ and clearly $(^0K^i)(K^i) = K^i$, hence (1) follows.

Let us prove (2). Recall that for $0 \leq i \leq d$,  $\rho_i = \rho_i' \otimes (\phi_i|_{K^i})$, so $\rho_i|_{^0K^i} = \rho_i'|_{^0K^i}\otimes (\phi_i|_{^0K^i})$.  Hence $m_{^0K^i}(\rho_i)= m_{^0K^i}(\rho_i')$.  To prove (2), it suffices to show that $m_{^0K^i}(\rho_i')=m_{^0K^{i-1}}( \rho_{i-1}')$. Note that $^0K^i = ^0K^{i-1}J^i$, and hence under the map \begin{align}\label{inflation}
    K^i\rightarrow K^{i-1}J^i/J^i = K^{i-1}/K^{i-1} \cap J^i,
\end{align} we have \begin{align}\label{inflation0}
    ^0K^i\rightarrow ^0K^{i-1}J^i/J^i = ^0K^{i-1}/^0K^{i-1} \cap J^i.
\end{align} 
Recall that $\mathrm{inf}(\rho_{i-1}')$ is the inflation of $\rho_{i-1}'$ to $K^i = K^{i-1} J^i$ via \eqref{inflation}. So the inflation  of $\rho_{i-1}'|_{^0K^{i-1}}$ to a representation of ${}^0 K^i$ via \eqref{inflation0}, denoted $\mathrm{inf}(\rho_{i-1}'|_{^0K^{i-1}})$,  is  precisely $\mathrm{inf}(\rho_{i-1}')|_{^0K^i}$ as representations of $^0K^i$. In particular,  \[\rho_i'|_{^0K^i} = \mathrm{inf}(\rho_{i-1}')|_{^0K^i} \otimes (\phi_{i-1}'|_{^0K^i}) = \mathrm{inf}(\rho_{i-1}'|_{^0K^{i-1}}) \otimes (\phi_{i-1}'|_{^0K^i}).\]
Next, we observe  that $(\phi_{i-1}'|_{^0K^i})$ is irreducible. This is in fact clear from Yu's construction of $\tilde\phi_{i-1}$;  in more detail,  Note that $^0K^i$  also acts on $J^i/J^i_+$ by conjugation and this gives a homomorphism from $^0K^i \rightarrow \text{Sp}(J^i/J^i_+)$.  Let $^0\tilde\phi_{i-1}$ be the pull back of the Weil representation of $\text{Sp}(J^i/J^i_+) \ltimes (J^i/J^i_+)^\#$ via the map $^0K^{i-1} \rtimes J^i \rightarrow \text{Sp}(J^i/J^i_+) \rtimes (J^i/J^i_+)^\#$. Then clearly, $(\tilde\phi_{i-1}|_{^0K^i}) = ^0\tilde\phi_{i-1}$ is irreducible. This proves that 
$\phi_{i-1}'|_{^0K^i}= \mathrm{inf}(\phi_{i-1})|_{^0K^i}\otimes (\tilde\phi_{i-1}|_{^0K^i}) =  \mathrm{inf}(\phi_{i-1})|_{^0K^i} \otimes  (^0\tilde\phi_{i-1})$ is irreducible. 

The construction of $\tilde\phi_{i-1}$ recalled above also shows that $\phi_{i-1}'|_{J^i}$ is irreducible. Now, let ${}^0\rho_0'$ be an irreducible summand of $\rho_0'|_{{}^0 K^0}$ and set ${}^0\rho_i' := \mathrm{inf}({}^0\rho_{i-1}') \otimes (\phi_{i-1}'|_{{}^0K^i})$ where $\mathrm{inf}({}^0\rho_{i-1}')$ is the representation of ${}^0 K^i$ via \eqref{inflation0}. Note that
\[\End_{\BC[{}^0 K^i]}({}^0\rho_i') = \End_{\BC[{}^0 K^{i-1}]}({}^0\rho_{i-1}')\] because $J^i$ acts trivially on $\mathrm{inf}({}^0\rho_{i-1}')$ and because $\phi_{i-1}'|_{J^i}$ is irreducible. Hence  ${}^0\rho_i' $ is an irreducible summand of $\rho_i'|_{{}^0K^i}$.  To finish the proof of (2), we need to show that 
\[\dim_\BC \Hom_{{}^0 K^i}({}^0\rho_i', \rho_i') =  \dim_\BC \Hom_{{}^0 K^{i-1}}({}^0\rho_{i-1}', \rho_{i-1}').\]
This again holds because $J^i$ acts trivially on $\mathrm{inf}({}^0\rho_{i-1}')$ and on $\mathrm{inf}(\rho_{i-1}')$ by \eqref{inflation} and \eqref{inflation0} respectively, and because $\phi_{i-1}'|_{J^i}$ is irreducible. 
\end{proof}

\begin{corollary}\label{CorHAYu}
Let $G$ be a connected, reductive group over $F$ and let $\pi_{-1} = \ind_{K^0}^{G^0} \rho_{-1}$ be the depth zero supercuspidal representation of $G^0$ that is part of the Yu datum. Let $\pi = \pi_d$ be a tame supercuspidal representation of $G$ of depth $r_d$, obtained as in Yu's construction above. Let $(^0K^d, ^0\rho_d)$ be the corresponding $[G, \pi_d]_{G}$-type described above. Then
\begin{enumerate}
    \item $\pi_{-1}|_{^0G^0}$ is multiplicity free if and only if $\pi_d|_{^0G^d}$ is multiplicity free.
    \item $ \CH(G^0, ^0\rho_{-1})$ is commutative if and only if $\CH(G, ^0\rho_d)$ is commutative.
\end{enumerate}

\end{corollary}
\begin{proof}
This corollary follows from Lemma \ref{LemmaonYu} and Proposition \ref{Prop:HAComm}. 
\end{proof}
We remark here that (2) of the preceding corollary can also be deduced as an obvious consequence of \cite[Theorem 4.5]{Oha21}. When one merely wants to compare the commutativity of these Hecke algebras, the above provides an alternate (elementary) argument. 
\subsection{The Bernstein center of $\CH(G,K)$}
We now assimilate the results in the preceding sections to describe the Bernstein center of $\CH(G, K)$ for $K \in \CK^\spade(S,G)$ and for $K \in \CK^\heart(S,G)$ via the theory of types.

\begin{corollary}\label{Cor1} Let $\bf S$ be a maximal $F$-split torus in $\bf G$ and let $K \in \CK^\spade(S,G)$. 
We have
\[\fkZ(\CH(G,K)) = \prod_{\fks \in \fkS(K)} \fkZ(\CH(G,\rho)) \simeq \prod_{\fks \in \fkS(K)} \BC[{}^\dag J_M/J_M]^{W(\rho_M)}\]
where, for $\fks = [M, \sigma]_G$, $(J_M, \rho_M)$ is a supercuspidal type of $\fks_M$, and $(J,\rho)$ is a $G$-cover of $(J_M, \rho_M)$.
\end{corollary}
\begin{proof}
This is a consequence of Theorem \ref{Satakegen}. Note that we can apply Theorem \ref{Satakegen} since the assumption $\CI_M(\rho_M) \subset \tJ_M$ is satisfied by Lemma \ref{LemmaonYu}.
\end{proof}

Now, we describe the center of $\CH(G,K)$ for $K \in \CK^\heart(S,G)$. Let $\bf M$ be a Levi subgroup of $\bf G$ that contains $\bf S$  and let 
\[l_{ M}^{G}: \fkB(M) \rightarrow \fkB(G), \;\;[M, \sigma]_M \rightarrow [M,\sigma]_G.\]
Let $\fkS(K \cap M)_{sc} = \{[M, \sigma]_M \in \fkS(K \cap M)\;|\;  \sigma \text{ a supercuspidal representation of $M$}\}$. 
\begin{corollary}\label{Cor2} Let $K \in \CK^\heart(S,G)$.  We have
\[\fkZ(\CH(G,K)) \simeq \prod_{[M]}\prod_{\fks_M \in \fkS(K\cap M)_{sc}} \BC[{}^\dag J_M/J_M]^{W(\rho_M)}\]
where $[\bf M]$ runs through the $\bf G$-conjugacy classes of $F$-Levi subgroups of $\bf G$ and $\bf M$ is a representative in this conjugacy class that contains $\bf S$. 
\end{corollary}
\begin{proof}
Note that by Lemma \ref{BuProp}, we have
\[\fkS(K) = \bigsqcup_M l_M^G(\fkS(K\cap M)_{sc}).  \]
Hence the corollary follows from Corollary \ref{Cor1}.
\end{proof}

\end{document}